\documentclass[11pt]{amsart}
\usepackage{fullpage,url,amssymb,mathrsfs,color,multicol,enumerate,moreverb}
\usepackage[all]{xy} 
\usepackage[OT2,T1]{fontenc}

\numberwithin{equation}{section}

\usepackage{color}

\newcommand{\blank}[1]{{}}

\def\bbar#1{\setbox0=\hbox{$#1$}\dimen0=.2\ht0 \kern\dimen0 }
\newcommand{\defi}[1]{\textsf{#1}} 

\newenvironment{romanenum}{\hfill \begin{enumerate} }{\end{enumerate}}

\DeclareSymbolFont{cyrletters}{OT2}{wncyr}{m}{n}
\DeclareMathSymbol{\Sha}{\mathalpha}{cyrletters}{"58}

\newcommand{\Aff}{{\mathbb A}}

  \newcommand{\FF}{{\mathbb F}}

\newcommand{\PP}{{\mathbb P}}
 \newcommand{\QQ}{{\mathbb Q}}

\newcommand{\ZZ}{{\mathbb Z}}

\newcommand{\kbar}{{\bbar{k}}}

\def\bbar#1{\setbox0=\hbox{$#1$}\dimen0=.2\ht0 \kern\dimen0 \overline{\kern-\dimen0 #1}}
\newcommand{\Qbar}{{\overline{\mathbb Q}}} 
\newcommand{\Kbar}{\bbar{K}} 
 
 \newcommand{\FFbar}{\overline{\FF}}

\newcommand{\calE}{{\mathcal E}} 
\newcommand{\calF}{{\mathcal F}}
\newcommand{\calG}{{\mathcal G}}
\newcommand{\calH}{{\mathcal H}}
\newcommand{\calI}{{\mathcal I}}

\newcommand{\calP}{{\mathcal P}}

\newcommand{\calT}{{\mathcal T}}
\newcommand{\calU}{{\mathcal U}}
\newcommand{\calV}{{\mathcal V}}
\newcommand{\calW}{{\mathcal W}}
\newcommand{\calX}{{\mathcal X}}
\newcommand{\calY}{{\mathcal Y}}

\newcommand{\OO}{{\mathcal O}}

\DeclareMathOperator{\tr}{tr}

\DeclareMathOperator{\Frob}{Frob}

\DeclareMathOperator{\End}{End}

\DeclareMathOperator{\Aut}{Aut}
\DeclareMathOperator{\Gal}{Gal}

\DeclareMathOperator{\NS}{NS}

\DeclareMathOperator{\Spec}{Spec}

\newcommand{\un}{{\operatorname{un}}}

\newcommand{\et}{{\operatorname{\acute{e}t}}}

\newcommand{\sss}{{\operatorname{ss}}}

\newcommand{\GL}{\operatorname{GL}}
\newcommand{\SL}{\operatorname{SL}}

\newcommand{\PSL}{\operatorname{PSL}}

\newcommand{\lcm}{\operatorname{lcm}}

\newcommand{\SO}{\operatorname{SO}}

\newcommand{\Or}{\operatorname{O}}

\DeclareMathOperator{\spin}{sp}

\newcommand{\ang}[2]{\langle #1,#2\rangle}

\def\CC{\mathbb C}
\def\p{\mathfrak{p}}

\newtheorem{theorem}{Theorem}[section]
\newtheorem{lemma}[theorem]{Lemma}

\newtheorem{proposition}[theorem]{Proposition}

\theoremstyle{definition}

\theoremstyle{remark}
\newtheorem{remark}[theorem]{Remark}

\definecolor{webcolor}{rgb}{0,0,1}
\definecolor{webbrown}{rgb}{.6,0,0}
\usepackage[
        colorlinks,
        linkcolor=webbrown,  filecolor=webbrown,  citecolor=webbrown, 
        backref,
        pdfauthor={David Zywina}, 
]{hyperref}
\usepackage[alphabetic,backrefs,lite]{amsrefs} 

\definecolor{backgroundcolor}{rgb}{1,0.9,1}

\begin{document}
\title[]{The inverse Galois problem for $\text{PSL}_2(\FF_p)$}
\author{David Zywina}
\address{School of Mathematics, Institute for advanced study, Princeton, NJ 08540}
\email{zywina@math.ias.edu}

\begin{abstract}
We show that the simple group $\PSL_2(\FF_p)$ occurs as the Galois group of an extension of the rationals for all primes $p\geq 5$.   We obtain our Galois extensions by studying the Galois action on the second \'etale cohomology groups of a specific elliptic surface.  
\end{abstract}

\maketitle

\section{Introduction}

\subsection{Statement}
The \defi{Inverse Galois Problem} asks whether every finite group $G$ occurs as the Galois group of an extension of $\QQ$, i.e., whether there is a finite Galois extension $K/\QQ$ such that $\Gal(K/\QQ)$ is isomorphic to $G$.  This problem is still wide open; even in the case of finite non-abelian simple groups which we now restrict our attention to.  Many special cases are known, including alternating groups and all but one of the sporadic simple groups.  Various families of simple groups of Lie type are known to occur as Galois groups of an extension of $\QQ$, but usually with congruences imposed on the cardinality of the fields.   See \cite{MR1711577} for background and many examples.  

We shall study the simple groups $\PSL_2(\FF_p):=\SL_2(\FF_p)/\{\pm I\}$ where $p\geq 5$ is a prime; their simplicity was observed by Galois, cf.~\cite[p.~411]{GaloisOeuvres}.  These are the ``simplest'' of the simple groups for which the inverse Galois problem has not been completely settled.  

Shih proved that the group $\PSL_2(\FF_p)$ occur as the Galois group of an extension of $\QQ$ if $2$, $3$ or $7$ is a quadratic non-residue modulo $p$, cf.~\cite{MR0332725} (see \S5.3 of \cite{MR2363329} for a lucid exposition).  The case where $5$ is a quadratic non-residue modulo $p$ is then a consequence of a theorem of Malle \cite{MR1199685}.    Clark \cite{MR2262856} showed that $\PSL_2(\FF_p)$ occurs for a finite number of additional primes $p$ (and also, assuming the Birch and Swinnerton-Dyer conjecture for elliptic curves over $\QQ$, primes $p\equiv 1\pmod{4}$ for which $11$ or $19$ is a quadratic non-residue modulo $p$).  Our main result is the following:

\begin{theorem} \label{T:main}
The simple group $\PSL_2(\FF_p)$ occurs as the Galois group of an extension of $\QQ$ for all primes $p\geq 5$.  
\end{theorem}

\begin{remark}
(ii) The Inverse Galois Problem is easy for the (non-simple) groups $\PSL_2(\FF_2)$ and $\PSL_2(\FF_3)$; note that they are isomorphic to $\mathfrak{S}_3$ and $\mathfrak{A}_4$, respectively.    

\noindent (i) In \cite{Zywina-orthogonal}, we will show that the simple groups $O_{2n+1}(p)$ and $O_{4n}^+(p)$ both occur as the Galois group of an extension of $\QQ$ for every prime $p\geq 5$ and integer $n\geq 2$ (with group notation as in \cite{Atlas}).   Moreover, these groups are shown to arise as the Galois group of a regular extension of the function field $\QQ(t)$.  This paper arose by trying to make progress in the excluded case $n=1$ where we have exceptional isomorphisms $O_3(p)\cong \PSL_2(\FF_\ell)$ and $O_{4}^+(p)\cong \PSL_2(\FF_p)\times \PSL_2(\FF_p)$. 
\end{remark}

\subsection{The representations} \label{SS:intro rep}
Let $\PP^1_\QQ$ be the projective line; it is obtained by completing the affine line $\Aff^1_\QQ:=\Spec \QQ[t]$ with a rational point $\infty$.   Consider the Weierstrass equation 
\begin{align} \label{E:main}
t(t-1)(t+1) \cdot y^2 = x(x+1)(x+t^2).
\end{align}
The equation (\ref{E:main}) defines a relative elliptic curve $f\colon E \to U$ where $U:=\PP^1_\QQ-\{0,1,-1,\infty\}$.   More precisely, $E$ is the projective subscheme of $\PP^2_U$ defined by the affine equation (\ref{E:main}) and $f$ is obtained by composing the inclusion $E\subseteq \PP^1_U$ with the structure map $\PP^1_U \to U$.  The points on $E$ off the affine model (\ref{E:main}) give our distinguished section of $f$.    For any field $K\supseteq \QQ$ and point $s \in U(K) \subseteq K$, the fiber $f^{-1}(s)$ is the elliptic curve over $K$ defined by specializing $t$ by $s$ in (\ref{E:main}).\\

Take any odd prime $\ell$.  Let $E[\ell]$ be the $\ell$-torsion subscheme of $E$.   The morphism $E[\ell]\to U$ arising from $f$ allows us to view $E[\ell]$ as a sheaf of $\FF_\ell$-modules on $U$; it is free of rank $2$.   Define the $\FF_\ell$-vector space
\[
V_\ell := H^1_{c}(U_\Qbar, E[\ell])
\]
where we are taking \'etale cohomology with compact support.  There is a natural action of the absolute Galois group $\Gal_\QQ:=\Gal(\Qbar/\QQ)$ on $V_\ell$ which we may express in terms of a representation
\[
\rho_\ell \colon \Gal_\QQ \to \Aut_{\FF_\ell}(V_\ell).
\]  
Theorem~\ref{T:main} will be a consequence of the following.
\begin{theorem} \label{T:main 2}
For each prime 	$\ell\geq 11$, the group $\PSL_2(\FF_\ell)$ is a quotient of $\rho_\ell(\Gal_\QQ)$.
\end{theorem}

In \S\ref{S:image}, we shall describe the group $\rho_\ell(\Gal_\QQ)$ for all $\ell \geq 11$.  Using the image of $\rho_\ell$, we will use Serre's modularity theorem in \S\ref{S:modularity} to show that our $\PSL_2(\FF_\ell)$-extensions arise from modular representations.

\subsection{The elliptic surface} \label{SS:intro surface}

 We can extend $f\colon E\to U$ to a morphism $\tilde{f}\colon X \to \PP^1_\QQ$ such that  $X$ is a smooth projective surface defined over $\QQ$ and $\tilde{f}$ is relatively minimal (so if $f'\colon X' \to \PP^1_\QQ$ was a morphism extending $f$ with $X'$ smooth and projective, then it would factor through $\tilde{f}$).   The surface $X$ is unique up to isomorphism.   

For each prime $\ell$, there is a natural Galois action on the \'etale cohomology group $H^2_{\et}(X_\Qbar,\FF_\ell)$ which can be expressed in terms of a representation 
\[
\phi_\ell\colon \Gal_\QQ \to \Aut_{\FF_\ell}(H^2_{\et}(X_{\Qbar},\FF_\ell)).   
\]
For odd $\ell$, we shall see that a certain Tate twist of $V_\ell$ is isomorphic to a composition factor of the $\FF_\ell[\Gal_\QQ]$-module $H^2_{\et}(X_{\Qbar},\FF_\ell)$; this will allow us to prove the following.

\begin{theorem} \label{T:main 3}
The group $\PSL_2(\FF_\ell)$ is a quotient of $\phi_\ell(\Gal_\QQ)$ for all $\ell \geq 11$.
\end{theorem}

\section{Basic properties and an overview} \label{S:basics}

Take any odd prime $\ell$ and fix notation as in \S\ref{SS:intro rep}.   We now describe many useful facts concerning the representation $\rho_\ell$ of \S\ref{SS:intro rep}; they are straightforward \'etale cohomology computations and we will supply proofs in Appendix~\ref{SS:appendix}.

\subsection{Properties}

\begin{lemma} \label{L:basics}
\begin{romanenum}
\item   \label{I:dim 4}
The $\FF_\ell$-vector space $V_\ell$ has dimension $4$.  
\item  \label{I:pairing}   
There is a non-degenerate symmetric bilinear form $\ang{\,}{\,}\colon V_\ell \times V_\ell \to \FF_\ell$ such that 
\[
\ang{\rho_\ell(\sigma) v}{\rho_\ell(\sigma)w} = \ang{v}{w}
\] 
for all $\sigma\in \Gal_\QQ$ and $v,w \in V_\ell$.   
\end{romanenum}
\end{lemma}

Let $\Or(V_\ell)$ be the group of automorphisms of the vector space $V_\ell$ which preserve the pairing $\ang{\,}{\,}$ of Lemma~\ref{L:basics}(\ref{I:pairing}).   We thus have a representation
\[
\rho_\ell\colon \Gal_\QQ \to \Or(V_\ell).
\]
Let $\SO(V_\ell)$ be the kernel of the determinant map $\det\colon \Or(V_\ell) \to \{\pm 1\}$.   The image of $\rho_\ell$ actually lies in this smaller group.

\begin{lemma} \label{L:SO image}
We have $\rho_\ell(\Gal_\QQ) \subseteq \SO(V_\ell)$.
\end{lemma}

\subsection{$L$-functions} \label{SS:Lfunctions}

Fix an odd prime $p$ and let $E_p$ be the elliptic curve over $\FF_p(t)$ defined by (\ref{E:main}).  Take any closed point $x$ of $\calU_p:=\PP^1_{\FF_p}-\{0,1,-1,\infty\}=\Spec \FF_p[t,1/(t(t-1)(t+1))]$.   Let $\FF_x$ be the residue field of $x$ and define $\deg x := [\FF_x:\FF_p]$.   Let $E_{p,x}$ be the elliptic curve over $\FF_x$ obtained by reducing $E_p$.   Let $a_x$ be the integer that satisfies $|E_{p,x}(\FF_x)|=|\FF_x| -a_x +1$.   The \defi{$L$-function} of the elliptic curve $E_p$ is the power series
\begin{equation} \label{L:L ps defn}
L(T,E_p) = {\prod}_x (1-a_x T^{\deg x} + p^{\deg x} T^{2\deg x})^{-1} \in \ZZ[\![T]\!]
\end{equation}
where the product is over the closed points $x$ of $\calU_p$ (we do not need to include factors at $0$, $1$, $-1$ and $\infty$ since $E_p$ has additive reduction at these points).     Define $P_p(T):= L(T/p, E_p)$.

\begin{lemma}   \label{L:deg 4}
For each odd prime $p$, $P_p(T)$ is a polynomial of degree $4$ with coefficients in $\ZZ[1/p]$ that satisfies $T^4 P_p(1/T) = P_p(T)$.
\end{lemma}

We will need to know the polynomials $P_p(T)$ for a few small primes $p$.   

\begin{lemma}
\label{L:explicit}
We have
$P_3(T)= 1 - 2/9\cdot T^2 + T^4$  and $P_5(T)= (1-2/5\cdot T+T^2)^2$.  
\end{lemma}

For each integer $n$ and odd prime $p$,  define the polynomial $P^{(n)}_p(T)= \prod_{i=1}^4(1-\alpha_i^n T) \in \QQ[T]$ where $P_p(T)=\prod_{i=1}^4(1-\alpha_i T)$ with $\alpha_i\in \Qbar$.  Using Lemma~\ref{L:explicit}, it is easy to verify that
\begin{align} \label{E:P4 explicit}
P_3^{(2)}(T)&=(1-2/9\cdot T+T^2)^2, \\[-0.3 em]
P_3^{(4)}(T)&=(1 + 158/81\cdot T + T^2)^2 \quad \text{ and }   \quad P_5^{(4)}(T)=(1 - 866/625\cdot T + T^2)^2.    \notag
\end{align}

\subsection{Compatibility} \label{SS:compatibility}

Here, and throughout the paper, $\Frob_p$ denotes an \emph{arithmetic} Frobenius automorphism in $\Gal_\QQ$ corresponding to $p$.  The following says that our representations $\rho_\ell$ are compatible and links them to the polynomials of \S\ref{SS:Lfunctions}.

\begin{lemma} \label{L:compatibility}
For each prime $p\nmid 2\ell$, the representation $\rho_\ell$ is unramified at $p$ and we have
\[
\det(I- \rho_\ell(\Frob_p) T) \equiv P_p(T) \pmod{\ell}.
\]
\end{lemma}

As a consequence of Lemma~\ref{L:compatibility}, we find that for each integer $n$ and prime $p\nmid 2\ell$ we have $\det(I- \rho_\ell(\Frob_p)^n T) \equiv P_p^{(n)}(T) \pmod{\ell}$.

\subsection{Connection with the surface $X$}

Let $X$ be the surface of \S\ref{SS:intro surface}.  We now related the $\FF_\ell[\Gal_\QQ]$-modules $V_\ell$ and $H^2_{\et}(X_\Qbar, \FF_\ell)$.   It will be convenient to work with the Tate twist $H^2_{\et}(X_\Qbar, \FF_\ell)(1)=H^2_{\et}(X_\Qbar, \FF_\ell(1))$.

\begin{lemma} \label{L:X connection}
Suppose $\ell\geq 7$.   The semi-simplification of $H^2_{\et}(X_\Qbar, \FF_\ell(1))$ as an $\FF_\ell[\Gal_\QQ]$-submodule is isomorphic to $V_\ell \oplus \FF_\ell^{30}$.
\end{lemma}

\subsection{Brief overview}
We now give some motivation for the rest of the paper; this will not be used later.  Fix a prime $\ell\geq 11$.  In light of Lemma~\ref{L:compatibility}, our first approach in trying to compute the image of $\rho_\ell$ was to compute $P_p(T)$ for many $p$.  The following proposition describes the pattern we observed; we will prove it in Propositions~\ref{P:P form mod 1} and \ref{P:P form mod 3}.

\begin{proposition} \label{P:pattern}
Let $p$ be an odd prime.
\begin{romanenum}
\item \label{P:pattern a}
If $p\equiv 1 \pmod{4}$, then $P_p(T)=(1+b T +T^2)^2$ for a unique $b\in \ZZ[1/p]$.
\item
If $p\equiv 3 \pmod{4}$, then $P_p(T)=1+(b^2-2) T^2 +T^4$ for a unique non-negative $b\in \ZZ[1/p]$.
\end{romanenum}
\end{proposition}

That the shape of $P_p(T)$ seems to depend on the value of $p$ modulo $4$ suggests that we study the action of $\Gal_{\QQ(i)}$ on $V_\ell$; this is done in \S\ref{S:over Q(i)}.   Using the automorphisms of the surface $X_{\QQ(i)}$, we will see that all the irreducible $\FF_\ell[\Gal_{\QQ(i)}]$-submodules of $V_\ell$ have dimension $1$ or $2$.  We will eventually show that $V_\ell$ is isomorphic to $W\oplus W$ as an $\FF_\ell[\Gal_{\QQ(i)}]$-module where $W$ is irreducible of dimension $2$.    

Let $\Omega(V_\ell)$ be the commutator subgroup of $\Or(V_\ell)$; it is an index $2$ subgroup of $\SO(V_\ell)$.    We will show that $\rho_\ell(\Gal_\QQ) \subseteq \Omega(V_\ell)$.  The group $\Omega(V_\ell)$ contains $-I$ and there is an \emph{exceptional isomorphism} $\varphi\colon \Omega(V_\ell)/\{\pm I\} \xrightarrow{\sim} \PSL_2(\FF_\ell) \times \PSL_2(\FF_\ell)$.  We thus have a representation
\[
\xymatrix{
\vartheta_\ell\colon \Gal_\QQ \ar[r]^-{\rho_\ell} & \Omega(V_\ell) \ar@{->>}[r] & \Omega(V_\ell)/\{\pm1\} \ar[r]^-{\varphi} & \PSL_2(\FF_\ell) \times \PSL_2(\FF_\ell) \ar[r]^-{pr} & \PSL_2(\FF_\ell)
}
\]
where $pr$ is one of the two projections.    Making appropriate choices of $\varphi$ and $pr$, for each prime $p\nmid 2\ell$ we will have 
\[
\tr( \vartheta_\ell(\Frob_p) ) = \pm b
\]
where $b$ is the value from Proposition~\ref{P:pattern} modulo $\ell$.

The main task of this paper is to show that $\vartheta_\ell\colon \Gal_\QQ \to \PSL_2(\FF_\ell)$ is surjective (this will be done in \S\ref{S:main proof}).   The proof is based on Serre's open image theorem for non-CM elliptic curves over number fields, cf.~\cite{MR0387283}.  We assume that the image of $\vartheta_\ell$ is contained in one of the maximal subgroups of $\PSL_2(\FF_\ell)$ and then try to obtain a contradiction.  To follow Serre's approach, we will need to understand the image under $\rho_\ell$ of the inertia subgroup at $\ell$ and $2$, see \S\ref{S:ramification at ell} and \S\ref{S:ramification at 2}.

One can similarly construct $\vartheta_7\colon \Gal_\QQ \to \PSL_2(\FF_7)$; however, it appears not to be surjective.   We will thus impose the condition $\ell\geq 11$ throughout much of the paper.

\section{Decomposition over $\QQ(i)$} \label{S:over Q(i)}

Fix a prime $\ell \geq 11$.  We now explain how $V_\ell$ breaks up into irreducible representations under the $\Gal_{\QQ(i)}$-action.  Fix an irreducible $\FF_\ell[\Gal_{\QQ(i)}]$-submodule $W$ of $V_\ell$ and set $n:= \dim_{\FF_\ell} W.$    Let $\beta\colon \Gal_{\QQ(i)} \to \Aut_{\FF_\ell}(W)\cong \GL_n(\FF_\ell)$ be the representation describing the Galois action on $W$.   We denote by $W^\vee$ the dual space of $W$ with its obvious $\Gal_{\QQ(i)}$-action.

\begin{proposition} \label{P:V semisimplification}
Fix notation as above and
 let $V_\ell^\sss$ be the semi-simplification of $V_\ell$ as an $\FF_\ell[\Gal_{\QQ(i)}]$-module.
\begin{romanenum}
\item \label{I:V semisimplification 0}
The integer $n$ is either $1$ or $2$.
\item \label{I:V semisimplification a}
If $n=1$, then $V_\ell^\sss \cong W\oplus W \oplus W^\vee \oplus W^\vee$ and $W\not\cong W^\vee$ as $\FF_\ell[\Gal_{\QQ(i)}]$-modules.
\item \label{I:V semisimplification b}
If $n=2$, then $V_\ell \cong W \oplus W$ and $W\cong W^\vee$ as $\FF_\ell[\Gal_{\QQ(i)}]$-modules.  
\end{romanenum}
\end{proposition}

The proposition will be proved by showing that the automorphisms of the surface $X_{\QQ(i)}$ constrains the image of $\rho_\ell$.

\begin{lemma} \label{L:group Q}
Let $Q$ be the group of quaternions $\{\pm 1, \pm i,\pm j,\pm k\}$.  There is a homomorphism $\varphi\colon Q \to \Aut_{\FF_\ell}(V_\ell)$ such that $\varphi(Q)$ commutes with $\rho_\ell(\Gal_{\QQ(i)})$ and $\varphi(-1)=-I$.  
\end{lemma}
\begin{proof}
By base extending, we may assume that the varieties $U$ and $E$ and the morphism $f\colon E\to U$ are all defined over $\QQ(i)$.  Let $\alpha_1$ and $\alpha_2$ be the automorphisms of $E$ that are given by $(x,y,t)\mapsto (x,iy,-t)$ and $(x,y,t)\mapsto (t^{-2}x, it^{-1}y , t^{-1})$, respectively.   For each $i\in \{1,2\}$, there is a unique isomorphism $g_i \colon U \to U$ such that $f\circ \alpha_i = g_i\circ f$.   Observe that the automorphisms $\alpha_i$ permutes the fibers  of $f\colon E \to U$; moreover, the maps between fibers are isomorphisms of elliptic curves.  We find that $\alpha_i$ induces an automorphism of the group variety $E[\ell]$ (though not as a sheaf over $U$).   We thus have a commutative diagram
\[
\xymatrix{
E[\ell] \ar[r]^{\stackrel{\alpha_i}{\sim}}  \ar[d]^{f} & E[\ell] \ar[d]^{f} \\
U \ar[r]^{\stackrel{g_i}{\sim}} & U.
}
\]
where the horizontal morphisms are isomorphisms.   The morphisms $\alpha_i$ and $g_i$ thus gives rise to a linear automorphism $\tilde\alpha_i$ of $V_\ell=H^1_c(U_\Qbar,E[\ell])$.

 Let $A$ be the subgroup of automorphisms of $E$ generated by $\alpha_1$ and $\alpha_2$.   The action of the $\alpha_i$ on $V_\ell$ thus give rise to a homomorphism $\varphi \colon A \to \Aut_{\FF_\ell}(V_\ell)$ satisfying $\varphi(\alpha_i)=\tilde\alpha_i^{-1}$.   The group $\varphi(A)$ commutes with $\rho_{\ell}(\Gal_{\QQ(i)})$ since the automorphisms $\alpha_i$ are defined over $\QQ(i)$.     One can readily verify the relations $\alpha_1^4 = \alpha_2^4=1$, $\alpha_1^2=\alpha_2^2\neq 1,$ and $\alpha_1 \circ \alpha_2 \circ \alpha_1^{-1} = \alpha_2^{-1}$, which is enough to ensure that $A$ is isomorphic to the group $Q:=\{\pm 1, \pm i,\pm j,\pm k\}$ of quaternions.    
 
Let $\iota$ be the automorphism of $E$ defined by $(x,y,t)\mapsto (x,-y,t)$; it is equal to $\alpha_1^2$ and $\alpha_2^2$.   We need to show that the induced automorphism $\tilde\iota := 	\varphi(\iota)$ of $V_\ell$ is $-I$.  We have $f\circ \iota = f$, so $\iota$ is an automorphism of the sheaf $E[\ell]$ on $U$.  Moreover, $\iota$ acts as $-I$ on $E[\ell]$.   Therefore, $\tilde{\iota}$ acts as $-I$ on $V_\ell$.	
\end{proof}

\begin{proof}[Proof of Proposition~\ref{P:V semisimplification}]

     Let $Z$ be the $\FF_\ell$-subalgebra of $\End_{\FF_\ell}(W)$ consisting of those endomorphisms that commute with the action of $\Gal_{\QQ(i)}$.   Since $W$ is irreducible, Schur's lemma implies that $Z$ is a finite division ring and is hence a (commutative) field.   
          
  Suppose that $W$ is stable under the action of the group $Q$ from Lemma~\ref{L:group Q}.  Let $\tilde{\varphi}\colon Q \to \Aut_{\FF_\ell}(W)$ be the corresponding representation.   Since $\tilde\varphi(-1)=-I$ and every non-trivial normal subgroup of $Q$ contains $-1$, we find that $\tilde\varphi(Q)$ is isomorphic to $Q$.   Therefore, $\tilde\varphi(Q)$ is a non-abelian subgroup of $Z^\times$.   However, this contradicts that $Z$ is a field.    
  
Therefore, $W$ is not stable under the action of $Q$.  So there is an element $\gamma \in \varphi(Q)$ such that $\gamma(W)\neq W$.   The vector space $\gamma(W)$ is stable under the action of $\Gal_{\QQ(i)}$ since $\varphi(Q)$ commutes with $\rho_\ell(\Gal_{\QQ(i)})$.    The automorphism $\gamma$ gives an isomorphism of $\FF_\ell[\Gal_{\QQ(i)}]$-modules from $W$ to $\gamma(W)$.   We have $\gamma(W)\cap W=0$ since $W$ is irreducible and $\gamma(W)\neq W$.   Therefore, $W\oplus\gamma(W)$ is a submodule of $V_\ell$ with $\gamma(W)$ isomorphic to $W$.   Since $V_\ell$ has dimension $4$, this proves that $n$ is $1$ or $2$.

If $n=2$, then $V_\ell = W \oplus \gamma(W)\cong W\oplus W$.  The $\FF_\ell[\Gal_{\QQ(i)}]$-module $W$ is isomorphic to its dual since $V_\ell$ is isomorphic to its dual by Lemma~\ref{L:basics}(\ref{I:pairing}).  
  
We now suppose that $n=1$.  There is a submodule of $V_\ell$ isomorphic to $W\oplus W$.   Since $V_\ell$ is self-dual by Lemma~\ref{L:basics}(\ref{I:pairing}), there is a submodule of $V_\ell^\sss$ isomorphic to $W^\vee \oplus W^\vee$.  To finish the proof of (\ref{I:V semisimplification a}), it suffices to show that $W\not\cong W^\vee$.    Assume to the contrary that $W \cong W^\vee$.   This implies that the corresponding character $\beta\colon \Gal_{\QQ(i)}\to \Aut_{\FF_\ell}(W)=\FF_\ell^\times$ and its inverse are equal.   Therefore, $\beta(\Gal_{\QQ(i)})\subseteq \{\pm 1\}$, and hence $\det(I+g)=0$ or $\det(I-g)=0$ for every $g\in \rho_\ell(\Gal_{\QQ(i)})$.  

The prime $5$ splits in $\QQ(i)$ and $\det(I-\rho_{\ell}(\Frob_5)T)\equiv P_5(T) \pmod{\ell}$, so $P_5(1)$ or $P_5(-1)$ must be divisible by $\ell$.  However, by Lemma~\ref{L:explicit} we have $P_5(1)=(8/5)^2$ and $P_5(1)=(12/5)^2$ which are not divisible by $\ell \geq 11$.  
\end{proof}

\section{Ramification at $\ell$} \label{S:ramification at ell}
Throughout this section, we fix a prime $\ell\geq 11$.

\subsection{Tame inertia}
Let $\Qbar_\ell$ be an algebraic closure  of $\QQ_\ell$.  Let $\QQ_\ell^{\un}$ be the maximal unramified extension of $\QQ_\ell$ in $\Qbar_\ell$.  Let $\QQ_\ell^t$ be the maximal tamely ramified extension of $\QQ_\ell$ in $\Qbar_\ell$.  The subgroup $\calI:=\Gal(\Qbar_\ell/\QQ_\ell^{\un})$ of $\Gal_{\QQ_\ell}$ is the \defi{inertia group}.   The \defi{wild inertia group} $\calP:=\Gal(\Qbar_\ell/\QQ_\ell^t)$ is the largest pro-$\ell$ subgroup of $\calI$.   The \defi{tame inertia group} is the quotient $\calI_t:= \calI/\calP$. 

For each positive integer $d$ relatively prime to $\ell$, let $\mu_d$ be the $d$-th roots of unity in $\Qbar_\ell$.   The map $\calI \to \mu_d$, $\sigma\mapsto \sigma(\sqrt[d]{\ell})/\sqrt[d]{\ell}$ is a surjective homomorphism which factors through $\calI_t$.   Taking the inverse limit over all $d$ relatively prime to $\ell$, we obtain an isomorphism $\calI_t \xrightarrow{\sim} \varprojlim_d \mu_d$.   The group $\mu_d$ lies in the ring of integers $\bbar{\ZZ}_\ell$ of $\Qbar_\ell$. Composing the homomorphism $\calI_t\to \mu_d$ with reduction modulo the maximal ideal of $\bbar{\ZZ}_\ell$ gives a character $\calI_t \to \FFbar_\ell^\times$.  For a positive integer $m$, setting $d:=\ell^m-1$ gives a surjective character $\varepsilon\colon \calI_t \to \FF_{\ell^m}^\times$.  We obtain $m$ characters $\calI_t\to \FF_{\ell^m}^\times$ by composing $\varepsilon$ with the isomorphisms of $\FF_{\ell^m}^\times$ arising from field automorphisms of $\FF_{\ell^m}$; they are called the \defi{fundamental characters} of level $m$.   See \S1 of \cite{MR0387283} for more details.

Let $V$ be an irreducible $\FF_\ell[\calI]$-module and set $m:=\dim_{\FF_\ell} V$.  The group $\calP$ act trivially on $V$, cf.~\cite{MR0387283}*{Proposition~4}.  Let $Z$ be the ring of endomomorphisms of $V$ as an $\FF_\ell[\calI_t]$-module.  Since $V$ is irreducible, $Z$ is a division algebra of finite dimension over $\FF_\ell$.  Therefore, $Z$ is a finite field and $V$ is a vector space of dimension $1$ over $Z$.   Choose an isomorphism $Z\cong \FF_{\ell^m}$ of fields.   The action of $\calI_t$ on $V$ corresponds to a character $\alpha\colon \calI_t \to Z^\times\cong \FF_{\ell^m}^\times$.  Let $\varepsilon_1,\ldots, \varepsilon_m \colon \calI_t \to \FF_{\ell^m}^\times$ be the fundamental characters of level $m$.   There are unique integers $e_i\in \{0,1,\ldots, \ell-1\}$ such that $\alpha= \varepsilon_1^{e_1}\cdots \varepsilon_m^{e_m}$.  These integers $e_1,\ldots,e_m$ are called the \defi{tame inertia weights} of $V$.         

For an $\FF_\ell[\calI]$-module $V$ of finite dimension over $\FF_\ell$, we define its tame inertia weights to be the integers that occur as the tame inertia weight of some composition factor of $V$.   \\

The following is a special case of a conjecture of Serre (cf.~\S1.13 of \cite{MR0387283}) and follows from a more general result of Caruso \cite{MR2372809}.

\begin{theorem}[Caruso] \label{T:Caruso main}
Let $\calX$ be a scheme that is proper and semistable over $\ZZ_\ell$.   For $i< \ell-1$, the tame inertia weights of $H^i_{\et}(\calX_{\Qbar_\ell}, \FF_\ell)^\vee$ belong to the set $\{0,1,\ldots, i\}$.
\end{theorem}

We will makes use of the following.

\begin{proposition} \label{P:Caruso}
The tame inertia weights of $H^2_{\et}(X_{\Qbar_\ell}, \FF_\ell)^\vee$ belong to the set $\{0,1,2\}$.
\end{proposition}
\begin{proof}
After a change of variable in (\ref{E:main}), we may start with the Weierstrass equation 
\[
y^2= x^3 + (t^5 - t)x^2 + (t^8 - 2t^6 + t^4)x;
\]
it has discriminant $16 t^{10}(t-1)^8(t+1)^8$.  Define $C:=\PP^1_{\ZZ_\ell}$ with a parameter $t$.  The Weierstrass equation defines a closed subscheme $\calY$ of $\PP^2_C$.  We have a morphism $\calY \to C$ obtained by composing the inclusion $\calY\subseteq \PP^2_C$ with the structure map $\PP^2_C\to C$.    One can check that the singular subscheme of $\calY$ is reduced and consists of the closure of the points $(0,0,0)$, $(0,0,1)$, $(0,0,-1)$ and a point with $t=\infty$.  By resolving the singularities appropriately (more explicitly, by following Tate's algorithm), we can construct a model $\calX/\ZZ_\ell$ of $X_{\QQ_\ell}$ that has good reduction.  The proposition then follows immediately from Theorem~\ref{T:Caruso main} (we have $2<\ell-1$ by our ongoing assumption $\ell \geq 11$).
\end{proof}

\subsection{Image of inertia} \label{SS:ramification at ell}
Let $\calI$ be an inertia subgroup of $\Gal_\QQ$ at $\ell$ and let $\calP$ be the wild inertia subgroup of $\calI$.  Since $\ell$ is unramified in $\QQ(i)$, the group $\calI$ is contained in $\Gal_{\QQ(i)}$.  Let $\chi_\ell \colon \Gal_\QQ \to \FF_\ell^\times$ be the representation describing the Galois action on the group of $\ell$-th roots of unity $\mu_\ell$, i.e., $\sigma(\zeta)=\zeta^{\chi_\ell(\sigma)}$ for $\sigma\in \Gal_\QQ$ and $\zeta\in \mu_\ell$.

Let $W$ be an irreducible $\FF_\ell[\Gal_{\QQ(i)}]$-submodule of $V_\ell$.  Set $n= \dim_{\FF_\ell} W$; it is $1$ or $2$ by Proposition~\ref{P:V semisimplification}(\ref{I:V semisimplification 0}).    Let $\beta\colon \Gal_{\QQ(i)} \to \Aut_{\FF_\ell}(W)\cong \GL_n(\FF_\ell)$ be the representation describing the Galois action on $W$.   The main task of this section is to prove the following proposition.

\begin{proposition} \label{P:ramification at ell} 
\begin{romanenum}
\item \label{P:ramification at ell a}
If $n=1$, then $\beta|_\calI = \chi_\ell^e|_\calI$ for some $e\in\{-1,0,1\}$.
\item \label{P:ramification at ell b}
If $n=2$, then $\beta(\calI)\subseteq \SL_2(\FF_\ell)$.
\item \label{P:ramification at ell c}
The group $\rho_\ell(\calI)/\rho_\ell(\calP)$ is cyclic of order $1$, $\ell-1$ or $\ell+1$.
\end{romanenum}
\end{proposition}

\begin{lemma} \label{L:beta image pm1} 
If $n=2$, then $\det(\beta(\Gal_\QQ)) \subseteq \{\pm 1\}$.
\end{lemma}
\begin{proof}
By Proposition~\ref{P:V semisimplification}(\ref{I:V semisimplification b}), $W$ is self-dual so $\det\circ \beta = (\det\circ \beta)^{-1}$.    The lemma is now immediate.
\end{proof}

\begin{lemma} \label{L:V semisimplification for inertia} 
Let $V_\ell^\sss$ be the semi-simplification of $V_\ell$ as an $\FF_\ell[\calI]$-module.   Let $\calW$ be an irreducible $\FF_\ell[\calI]$-submodule of $V_\ell^\sss$.  Set $m:=\dim_{\FF_\ell} \calW$ and let $\alpha\colon \calI \to \Aut_{\FF_\ell}(\calW) \cong \GL_m(\FF_\ell)$ be the representation corresponding to $\calW$.
\begin{romanenum}
\item \label{I:V semisimplification for inertia a} 
We have $m\leq n$.  In particular, $m$ is $1$ or $2$.
\item \label{I:V semisimplification for inertia d} 
If $m=1$, then $\alpha= \chi_\ell^e|_\calI$ for some $e\in\{-1,0,1\}$.
\item \label{I:V semisimplification for inertia e} 
If $m=2$, then $\alpha(\calI)$ is a cyclic group of order $\ell+1$ contained in $\SL_2(\FF_\ell)$.
\end{romanenum}
\end{lemma}

\begin{proof}
We have $\calI \subseteq \Gal_{\QQ(i)}$ since $\ell$ is unramified in $\QQ(i)$.  Part (\ref{I:V semisimplification for inertia a}) now follows immediately from Proposition~\ref{P:V semisimplification}.

Let $\calH$ be the semi-simplification of $H^2_{\et}(X_\Qbar, \FF_\ell(1))$ as an $\FF_\ell[\calI]$-module.  By Lemma~\ref{L:X connection},  we find that $\calW$ is isomorphic to a submodule of $\calH$.  Therefore, $\calW(-1)^\vee$ is isomorphic to a submodule of $\calH(-1)^\vee$.   So $\calW(-1)^\vee$ is isomorphic to a submodule of the semi-simplification of $H^2_{\et}(X_\Qbar, \FF_\ell)^\vee$ as an $\FF_\ell[\calI]$-module.  Proposition~\ref{P:Caruso} implies that the possible tame inertia weights of $\calW(-1)^\vee$ are $0$, $1$ and $2$.

First suppose that $m=1$.   We then have $(\alpha \cdot \chi_\ell|_\calI^{-1})^{-1} = \varepsilon^e$ for some $e\in\{0,1,2\}$ where $\varepsilon\colon \calI\to \FF_\ell^\times$ is the fundamental character of level $1$.  By Proposition~8 of \cite{MR0387283}, we have $\varepsilon=\chi_\ell|_{\calI}$.  Therefore, $\alpha= \varepsilon^{-e} \chi_\ell|_\calI = \chi_\ell|_\calI^f$ where $f:=1-e \in \{-1,0,1\}$.  This proves (\ref{I:V semisimplification for inertia d}).

Now suppose that $m=2$.  The irreducible representation $\alpha$ corresponds to a character $\calI \to \FF_{\ell^2}^\times$ that we also denote by $\alpha$.   Let $\varepsilon$ be a fundamental character of level $2$.  Using Proposition~\ref{P:V semisimplification}, we find that $\calW$ is self-dual.   Therefore, $\calW(1)$ is isomorphic to $\calW(-1)^\vee$ as an $\FF_\ell[\calI]$-module, and hence has possible tame inertia weights $0$, $1$ and $2$.   So $\alpha=\chi_\ell|_{\calI}^{-1}\cdot \varepsilon^{e_1+\ell e_2}$ with $e_1,e_2\in\{0,1,2\}$.  The character $\varepsilon^{1+\ell}$ is fundamental of level $1$ and hence agrees with $\chi_\ell|_{\calI}$, so $\alpha = \varepsilon^{f_1 + f_2 \ell}$ with $f_i := 1-e_i \in \{-1,0,1\}$.   

If $f_1+f_2\ell \in \{0, \pm(\ell+1)\}$, then the character $\alpha$ is $1$, $\chi_\ell |_\calI$ or $\chi_\ell^{-1} |_{\calI}$; this is impossible since $\alpha\colon \calI\to \GL_2(\FF_\ell)$ is irreducible. 

If $f_1+f_2 \ell \in \{ \pm 1, \pm \ell\}$, then $\alpha(\calI)$ is cyclic of order $\ell^2-1$ since $\varepsilon(\calI) = \FF_{\ell^2}^\times$.  We have $n=2$ since $m=2$.    By Proposition~\ref{P:V semisimplification}, we deduce that $\alpha$ and $\beta|_\calI$ are isomorphic representations (we again have $\calI\subseteq \Gal_{\QQ(i)}$ since $\ell$ is unramified in $\QQ(i)$).   Therefore, $\beta(\Gal_{\QQ(i)})$ contains a cyclic group $C$ of order $\ell^2-1$.  By \cite{MR0387283}*{\S2.6}, $C$ is a split Cartan subgroup of $\GL_2(\FF_\ell)$ and satisfies $\det(C)=\FF_\ell^\times$.    So $\det(\beta(\Gal_{\QQ(i)}))=\FF_\ell^\times$ which contradicts Lemma~\ref{L:beta image pm1} since $\ell\geq 11$.

Therefore, $f_1+f_2\ell \in \{ \pm (\ell-1) \}$.  This implies that $\alpha(\calI)$ is cyclic of order $\ell+1$ and is contained in $\SL_2(\FF_\ell)$.  
\end{proof}

\begin{proof}[Proof of Proposition~\ref{P:ramification at ell}]
We have $\calI \subseteq \Gal_{\QQ(i)}$ since $\ell$ is unramified in $\QQ(i)$.   Let $\calW$ be an irreducible $\FF_\ell[\calI]$-submodule of $W$ and take $\alpha$ and $m$ as in Lemma~\ref{L:V semisimplification for inertia}.

Suppose that $n=1$ and hence $\calW=W$.    Part (\ref{P:ramification at ell a}) now follows directly from Lemma~\ref{L:V semisimplification for inertia}(\ref{I:V semisimplification for inertia d}).   By part (\ref{P:ramification at ell a}) and Proposition~\ref{P:V semisimplification}(\ref{I:V semisimplification a}), we deduce that the group $\rho_\ell(\calI)/\rho_\ell(\calP)$ is cyclic of order $1$ or $\ell-1$.      

We are left to consider the case $n=2$.  First suppose that $\calW=W$, and hence we may assume that $\beta|_\calI = \alpha$.    By Lemma~\ref{L:V semisimplification for inertia}(\ref{I:V semisimplification for inertia e}), $\beta(\calI)=\alpha(\calI)$ is a cyclic group of order $\ell+1$ in $\SL_2(\FF_\ell)$.    By Proposition~\ref{P:V semisimplification}(\ref{I:V semisimplification b}), $\rho_\ell(\calI)$ is isomorphic to $\beta(\calI)$ and hence is cyclic of order $\ell+1$.

The final case is where $n=2$ and $\dim_{\FF_\ell} \calW = 1$.  Let $W^\sss$ be the semi-simplification of $W$ as an $\FF_\ell[\calI]$-module.   There is an $\FF_\ell[\calI]$-module $\calW'$ such that $W^\sss \cong \calW \oplus \calW'$.   Let $\gamma\colon \calI \to \FF_\ell^\times$ and $\gamma'\colon \calI \to \FF_\ell^\times$ be the characters describing the action of $\calI$ on $\calW$ and $\calW'$, respectively.   By Lemma~\ref{L:V semisimplification for inertia}(\ref{I:V semisimplification for inertia d}), there are $e$ and $f$ in $\{-1,0,1\}$ such that $\alpha=\chi_\ell^e |_\calI$ and $\alpha'=\chi_\ell^{-f} |_\calI$.  So for $\sigma \in \calI$, we have 
\[
\det(I-\beta(\sigma) T) = (1-\alpha(\sigma) T)(1-\alpha'(\sigma) T)=  1- (\chi_\ell(\sigma)^e + \chi_\ell(\sigma)^{-f})T + \chi_\ell(\sigma)^{e-f} T^2.
\]
If $e\neq f$, then we find that $\det(\beta(\calI))$ is a cyclic group of cardinality at least $(\ell-1)/2 > 2$.  However this contradicts Lemma~\ref{L:beta image pm1}, so $e=f$.  Therefore, $\det(I-\beta(\sigma) T) = 1- (\chi_\ell(\sigma)^e + \chi_\ell(\sigma)^{-e})T + T^2$ for all $\sigma \in \calI$.  This proves that $\det(\beta(\calI))=\{1\}$, i.e., $\beta(\calI)\subseteq \SL_2(\FF_\ell)$.    We have $\beta(\calI)/\beta(\calP) \cong \alpha(\calI)=\chi^e(\calI)$, so $\beta(\calI)/\beta(\calP)$ is a cyclic group of order $1$ or $\ell-1$ if $e=0$ or $e=\pm 1$, respectively.  By Proposition~\ref{P:V semisimplification}(\ref{I:V semisimplification a}), we deduce that $\rho_\ell(\calI)/\rho_\ell(\calP)$ is also cyclic of order $1$ or $\ell-1$.
\end{proof}  

When $n=2$ we have the following important constraint on the image of $\beta$.

\begin{lemma} \label{L:beta image eq1}
If $n=2$, then $\beta(\Gal_{\QQ(i)}) \subseteq \SL_2(\FF_\ell)$.
\end{lemma}
\begin{proof}
By Lemma~\ref{L:beta image pm1}, we can define a character $\varepsilon\colon \Gal_{\QQ(i)} \to \{\pm 1\}$, $\sigma\mapsto \det(\beta(\sigma))$.  We need to show that $\varepsilon=1$.  By Proposition~\ref{P:V semisimplification}(\ref{I:V semisimplification b}), we have $V_\ell \cong W \oplus W$ as $\FF_\ell[\Gal_{\QQ(i)}]$-modules and hence
\begin{equation} \label{E:factorization for beta}
\det(I-\rho_\ell(\sigma) T) = \det(I-\beta(\sigma) T)^2 = (1- \tr(\beta(\sigma)) T + \epsilon(\sigma)  T^2)^2
\end{equation}
for $\sigma\in \Gal_{\QQ(i)}$.  Since $5$ splits in $\QQ(i)$, $\Frob_5$ and $\Frob_3^2$ belong to $\Gal_{\QQ(i)}$.   
By (\ref{E:factorization for beta}) and Lemma~\ref{L:explicit}, we have:
\begin{align*}
(1- \tr(\beta(\Frob_5)) T + \det( \beta(\Frob_5) ) T^2)^2 &\equiv P_5(T) = (1-2/5\cdot T+T^2)^2 \pmod{\ell}\\
(1- \tr(\beta(\Frob_3^2)) T + \det( \beta(\Frob_3^2) ) T^2)^2 &\equiv P_3^{(2)}(T) = (1-2/9\cdot T+T^2)^2 \pmod{\ell}.
\end{align*}
From the unique factorization of $\FF_\ell[T]$, we deduce that $\varepsilon(\Frob_3^2)=1$ and $\varepsilon(\Frob_5)=1$.  

We claim that the character $\varepsilon$ is unramified at all finite places $v$ of $\QQ(i)$ that do not lie over $2$.   Fix a finite place $v$ of $\QQ(i)$ that does not lie over $2$.   If $v$ does not lie over $\ell$, then $\varepsilon$ is unramified at $v$ since $\rho_\ell$ is unramified at primes $p\nmid 2\ell$.  So suppose that $v$ lies over $\ell$.   Let $\calI$ be an inertia subgroup of $\Gal_{\QQ(i)}$ at $v$.  Since $\ell$ is unramified in $\QQ(i)$, the group $\calI$ is also an inertia subgroup of $\Gal_\QQ$ at $\ell$.   By Proposition~\ref{P:ramification at ell}(\ref{P:ramification at ell b}), we have $\beta(\calI) \subseteq \SL_2(\FF_\ell)$ and hence $\varepsilon(\calI)=1$.

Now let $K$ be the fixed field in $\Qbar$ of $\ker(\varepsilon)$.  We have $[K:\QQ(i)]\leq 2$ and the extension $K/\QQ(i)$ is unramified at all finite places not lying over $2$.   Since $\ZZ[i]$ is a PID, $K$ can thus be obtained by adjoining to $\QQ(i)$ the square root of a squarefree element $a\in \ZZ[i]$ that is not divisibly by any prime of $\ZZ[i]$ except $1+i$.    So $a$ is of the form $\pm i^e (1+i)^f$ with $e,f\in \{0,1\}$.   Therefore, $K$ is the field obtained by adjoining to $\QQ(i)$ the square-root of some $a\in \{1, i, 1+i, i(1+i)\}$.

Let $\p_5$ be the prime ideal of $\ZZ[i]$ generated by $2+i$.    We have $\varepsilon(\Frob_{\p_5})=\varepsilon(\Frob_5)=1$ and hence $\p_5$ splits in $K$.   Therefore, $a$ modulo $\p_5$ is a square.    We have $i \equiv -2 \pmod{\p_5}$ and $i(1+i)\equiv 2 \pmod{\p_5}$.   Since $2$ and $-2$ are not squares in $\ZZ[i]/\p_5\cong \FF_5$, we deduce that $a\in \{1, 1+i\}$.

Let $\p_3$ be the prime ideal of $\ZZ[i]$ generated by $3$.   We have $\varepsilon(\Frob_{\p_3})=\varepsilon(\Frob_3^2)=1$ and hence $\p_3$ splits in $K$.  Therefore, $a$ modulo $\p_3$ is a square.   One can check that the image of $1+i$ modulo $\p_3$ generates the group $(\ZZ[i]/\p_3)^\times$ which is a cyclic group of order $8$; in particular, $1+i$ is not a square modulo $\p_3$.   Therefore, $a=1$.   So $K=\QQ(i)$ and hence $\varepsilon =1$.
\end{proof}

\subsection{$L$-functions with $p\equiv 1 \pmod{4}$} \label{SS:L with p eq 1 mod 4}

In this section, we show that the polynomial $P_p(T)$ with $p\equiv 1\pmod{4}$ are of a special form; we will consider the primes $p\equiv 3\pmod{4}$ in \S\ref{SS:L with p eq 3 mod 4}.

\begin{proposition}  \label{P:P form mod 1}
For each prime $p\equiv 1 \pmod{4}$, we have $P_p(T)=(1+b T +T^2)^2$ for a unique $b\in \ZZ[1/p]$.
\end{proposition}

In terms of our representations $\rho_\ell$, we have the following:
\begin{lemma} \label{L:square}
For every $\sigma\in \Gal_{\QQ(i)}$, we have $\det(I - \rho_\ell(\sigma) T ) = (1+bT +T^2)^2$ for a unique $b\in \FF_\ell$.
\end{lemma}
\begin{proof}   Fix notation as in the beginning of \S\ref{S:over Q(i)} and take any $\sigma\in \Gal_{\QQ(i)}$.   If $n=1$, then Proposition~\ref{P:V semisimplification}(\ref{I:V semisimplification a}) implies that
\[
\det(I-\rho_\ell(\sigma)T) = (1- \beta(\sigma) T)^2 (1- \beta(\sigma)^{-1} T)^2 = (1 - (\beta(\sigma)+\beta(\sigma)^{-1}) T +T^2)^2.
\]
This proves (i) in the case $n=1$; the uniqueness follows from unique factorization. 

We now assume that $n=2$.  By Proposition~\ref{P:V semisimplification}(\ref{I:V semisimplification b}), we have $V_\ell \cong W \oplus W$ as $\FF_\ell[\Gal_{\QQ(i)}]$-modules.  Therefore,
\[
\det(I-\rho_\ell(\sigma) T) = \det(I-\beta(\sigma) T)^2 = (1- \tr(\beta(\sigma)) T + T^2)^2
\]
where the last equality uses Lemma~\ref{L:beta image eq1}.
\end{proof}

\begin{proof}[Proof of Propositions~\ref{P:P form mod 1}]
By Lemma~\ref{L:square}, for each odd prime $p$ and prime $\ell\geq 11$ with $\ell\neq p$, we have
\[
P_p(T) \equiv \det(I - \rho_\ell(\Frob_p) T) = (1+b T + T^2)^2 \pmod{\ell}
\]
for some $b \in \FF_\ell$.   Since $P_p(T)$ modulo $\ell$ is of the form $(1+b T + T^2)^2$ for all but finitely many primes $\ell$, we deduce that $P_p(T)$ is of the form $(1+b T +T^2)^2$ for some $b\in \QQ$.     Therefore, $L(T,E_p)= P_p(pT)=(1+ bp T + p^2 T^2)^2$.  Since $L(T,E_p)$ has integer coefficients, unique factorization shows that $b$ is unique and that $bp \in \ZZ$.  This completes the proof of Proposition~\ref{P:P form mod 1}.
\end{proof}

\section{Orthogonal groups}
Throughout this section, we fix a prime $\ell\geq 11$.

\subsection{The group $\Omega(V_\ell)$}
  Let $\spin \colon \Or(V_\ell) \to \FF_\ell^\times/(\FF_\ell^\times)^2$ be the \defi{spinor norm}, cf.~\cite{MR0148760}.      If $v \in V_\ell$ satisfies $\ang{v}{v} \neq 0$, let $r_v \in \Or(V_\ell)$ be the reflection across $v$ (we have $r_v(v)=-v$ and $r_v(w)=w$ for all $w\in V_\ell$ such that $\ang{v}{w}=0$).   The spinor norm can be characterized by the property that it is the group homomorphism which satisfies $\spin(r_v) = \ang{v}{v} \cdot (\FF_\ell^\times)^2$ for all $v\in V_\ell$ for which $\ang{v}{v}\neq 0$.   
  
 We will use the following to compute spinor norms; it follows from equation (2.1) of \cite{MR0148760}*{\S2} and uses that $V_\ell$ has dimension $4$.

\begin{lemma} \label{L:spinor}
If $A\in \Or(V_\ell)$ satisfies $\det(I+A)\neq 0$, then $\spin(A)=\det(I+A) \cdot (\FF_\ell^\times)^2$.   \qed
\end{lemma}

Let $\Omega(V_\ell)$ be the simultaneous kernel of $\det\colon \Or(V_\ell) \to \{\pm 1\}$ and $\spin \colon \Or(V_\ell) \to \FF_\ell^\times/(\FF_\ell^\times)^2$.     We could also define $\Omega(V_\ell)$ as the commutator subgroup of $\Or(V_\ell)$.   We now show that the image of $\rho_\ell$ lies in this smaller group.

\begin{lemma} \label{L:main containment}
We have $-I \in \Omega(V_\ell)$ and $\rho_\ell(\Gal_\QQ) \subseteq \Omega(V_\ell)$.
\end{lemma}
\begin{proof}
We have $\rho_\ell(\Gal_\QQ) \subseteq \SO(V_\ell)$ by Lemma~\ref{L:SO image} and $\det(-I)=(-1)^4=1$.   It thus remains to show that $\spin(-I)=(\FF_\ell^\times)^2$ and that $\spin(\rho_\ell(\sigma))=(\FF_\ell^\times)^2$ for all $\sigma\in \Gal_\QQ$.

By Lemma~\ref{L:explicit}, we have $\det(I-\rho_\ell(\Frob_3)T)\equiv P_3(T) = 1 - 2/9\cdot T^2 + T^4 \pmod{\ell}$.   Since $\det(I\pm \rho_\ell(\Frob_3)) \equiv P_3(\pm 1) = (4/3)^2 \pmod{\ell}$, Lemma~\ref{L:spinor} implies that  $\spin(\rho_\ell(\Frob_3))$ and $\spin(-\rho_\ell(\Frob_3))$ both equal $(\FF_\ell^\times)^2$.   Therefore, $\spin(-I)= \spin(\rho_\ell(\Frob_3))\cdot \spin(-\rho_\ell(\Frob_3))= (\FF_\ell^\times)^2$.    Since $3$ is inert in $\QQ(i)$ and $\spin(\rho_\ell(\Frob_3))=(\FF_\ell^\times)^2$, it suffices to show that $\spin(\rho_\ell(\sigma))=(\FF_\ell^\times)^2$ for all $\sigma \in \Gal_{\QQ(i)}$.   

Take any $\sigma \in \Gal_{\QQ(i)}$.   By Lemma~\ref{L:square}, we have $\det(I - \rho_\ell(\sigma) T ) = (1+bT +T^2)^2$ for a unique $b\in \FF_\ell$.  If $b\neq 2$, then by Lemma~\ref{L:spinor} we have 
\[
\spin(\rho_\ell(\sigma))= \det(I+\rho_\ell(\sigma))\cdot (\FF_\ell^\times)^2 = (2-b)^2 \cdot (\FF_\ell^\times)^2 = (\FF_\ell^\times)^2.
\]
So suppose that $b=2$, and hence $\det(I - \rho_\ell(\sigma) T ) = (1+2T +T^2)^2 = (1+T)^4$.   Since $\ell$ is odd, there is an $e\geq 0$ such that $\rho_\ell(\sigma)^{\ell^e} = -I$.    Therefore, $\spin(\rho_\ell(\sigma))=\spin(\rho_\ell(\sigma))^{\ell^e} = \spin(-I) = (\FF_\ell^\times)^2$.
\end{proof}

\newpage

\subsection{The representation $\vartheta_\ell$} \label{SS:psi}

Define the $4$-dimensional $\FF_\ell$-vector space $\calV_\ell := \FF_\ell^2 \otimes_{\FF_\ell} \FF_\ell^2$.    We have a natural action of $\SL_2(\FF_\ell)\times \SL_2(\FF_\ell)$ on $\calV_\ell$ with $(-I,-I)$ acting trivially; we denote the image in $\Aut_{\FF_\ell}(\calV_\ell)$ by $\SL_2(\FF_\ell) \otimes \SL_2(\FF_\ell)$.

\begin{lemma} \label{L:new psi}
There is an isomorphism $V_\ell \cong \calV_\ell$ of vector spaces such that the induced isomorphism $\Aut_{\FF_\ell}(V_\ell) \cong \Aut_{\FF_\ell}(\calV_\ell)$ of groups gives an isomorphism $\psi_\ell \colon \Omega(V_\ell) \xrightarrow{\sim} \SL_2(\FF_\ell) \otimes \SL_2(\FF_\ell)$.
\end{lemma}
\begin{proof}
Let $\{e,f\}$ be the standard basis of the vector space $\calW:=\FF_\ell^2$ over $\FF_\ell$.   Let $h\colon \calW\times \calW \to \FF_\ell$ be the alternating bilinear pairing that satisfies $h(e,f)=1$.   The group of automorphisms of the vector space $\calW$ that respect the pairing $h$ is $\SL(\calW)=\SL_2(\FF_\ell)$.      Let $b$ be the bilinear pairing on $\calV_\ell=\calW \otimes_{\FF_\ell} \calW$ that satisfies  $b(v_1 \otimes w_1, v_2\otimes w_2)= h(v_1,v_2) h(w_1,w_2)$.  One can show that $b$ is symmetric and non-degenerate.    

As before, we can define $\Or(\calV_\ell)$ to be the group of automorphisms of $\calV_\ell$ that preserve the pairing $b$, and we define $\Omega(\calV_\ell)$ to be the simultaneous kernels of the determinant $\det\colon \Or(\calV_\ell)\to\{\pm 1\}$ and the spinor norm $\spin_{\calV_\ell}\colon \Or(\calV_\ell)\to \FF_\ell^\times/(\FF_\ell^\times)^2$.     By \cite{MR0148760}*{(2.3)}, $\spin_{\calV_\ell}(-I)$ agrees with the discriminant of the pairing $b$ on $\calV_\ell$.   One can then check that $\spin_{\calV_\ell}(-I)=(\FF_\ell^\times)^2$. 

Up to isomorphism, there are two $4$-dimensional vector spaces over $\FF_\ell$ equipped with a bilinear, symmetric and non-degenerate pairing. They can be distinguished by the spinor norm of $-I$ (see \cite{MR0344216}*{IV~\S1.7} where it is stated in terms of quadratic forms; the discriminant of the corresponding quadratic form agrees with the spinor norm of $-I$).  We have $\dim_{\FF_\ell} \calV_\ell = 4 = \dim_{\FF_\ell} V_\ell$ and $\spin_{\calV_\ell}(-I)=(\FF_\ell^\times)^2=\spin(-I)$ where the last equality uses Lemma~\ref{L:main containment}.  There is thus an isomorphism $\varphi\colon \calV_\ell \to V_\ell$ of $\FF_\ell$-vector spaces such that $\ang{\varphi(v)}{ \varphi(w)}= b(v,w)$ for all $v,w \in \calV_\ell$.   It now suffices to prove the lemma for $\Omega(\calV_\ell)$ instead of $\Omega(V_\ell)$.   

Since $\spin_{\calV_\ell}(-I)=(\FF_\ell^\times)^2$, $\Omega(\calV_\ell)$ is isomorphic to the group denoted by $\Omega_{4}^+(\ell)$ in \cite{Atlas}*{\S2}.    There is an exceptional isomorphism $\Omega_4^+(\ell)/\{\pm I\}\cong \PSL_2(\FF_\ell)\times \PSL_2(\FF_\ell)$, so $\Omega(\calV_\ell)/\{\pm I\}$ is isomorphic to $\PSL_2(\FF_\ell)\times \PSL_2(\FF_\ell)$.  We have a natural action of $\SL_2(\FF_\ell)\times \SL_2(\FF_\ell)$ on $\calV_\ell$ with $(-I,-I)$ acting trivially.  This action respects the pairing $b$ and gives rise to an injective homomorphism
\begin{equation} \label{E:xi}
\xi\colon\SL_2(\FF_\ell) \otimes \SL_2(\FF_\ell) =  (\SL_2(\FF_\ell) \times \SL_2(\FF_\ell))/\{\pm(I,I)\}  \hookrightarrow \Omega(\calV_\ell).
\end{equation}
Quotienting out by the subgroup generated by $(I,-I)$, $\xi$ gives rise to an injective homomorphism $\bbar\xi\colon \PSL_2(\FF_\ell)\times \PSL_2(\FF_\ell) \hookrightarrow \Omega(\calV_\ell)/\{\pm I\}$ that must be an isomorphism by cardinality considerations.  That $\bbar\xi$ is an isomorphism implies that $\xi$ is also an isomorphism.
\end{proof}

Let  $\widetilde\rho_\ell \colon \Gal_\QQ \to \SL_2(\FF_\ell) \otimes \SL_2(\FF_\ell)$ be the representation obtained by composing $\rho_\ell$ with the isomorphism $\psi_\ell$ of Lemma~\ref{L:new psi} (this of course uses Lemma~\ref{L:main containment}).   Let 
\[
\vartheta_\ell \colon \Gal_\QQ \to \PSL_2(\FF_\ell)
\] 
be the homomorphism obtained by composing $\widetilde\rho_\ell$ with the homomorphism $\SL_2(\FF_\ell) \otimes \SL_2(\FF_\ell) \to \PSL_2(\FF_\ell)$ which maps $A\otimes B$ to the image of $A$ in $\PSL_2(\FF_\ell)$.   Similarly, we define $\vartheta_\ell' \colon \Gal_\QQ \to \PSL_2(\FF_\ell)$ except projecting on the second factor.

\begin{lemma} \label{L:need to switch}
After possibly switching $\vartheta_\ell$ and $\vartheta_\ell'$, we may assume that the group $\vartheta_\ell'(\Gal_{\QQ(i)})$ has cardinality $1$ or $\ell$.  For $\sigma\in \Gal_{\QQ(i)}$, we have $\tr(\vartheta_\ell(\sigma))=\pm b$ where $\det(I-\rho_\ell(\sigma)T)=(1+bT+T^2)^2$.
\end{lemma}
\begin{proof}
Take any $\sigma\in \Gal_{\QQ(i)}$.    Choose $A,B \in \SL_2(\FF_\ell)$ such that $\psi_\ell(\rho_\ell(\sigma)) = \widetilde\rho_\ell(\sigma)$ equals $A\otimes B$.  Take $\lambda_1, \lambda_2\in \FFbar_\ell^\times$ such that the eigenvalues of $A$ and $B$ are $\{\lambda_1,\lambda_1^{-1}\}$ and $\{\lambda_2,\lambda_2^{-1}\}$, respectively.  The roots of $\det(I-\rho_\ell(\sigma)T)=\det(I-\widetilde\rho_\ell(\sigma)T)$ in $\FFbar_\ell$ are thus $\lambda_1 \lambda_2$, $\lambda_1\lambda_2^{-1}$, $\lambda_1^{-1}\lambda_2$ and $\lambda_1^{-1}\lambda_2^{-1}$.  

We claim that the image of $A$ or $B$ in $\PSL_2(\FF_\ell)$ has order belonging to $\{1,\ell\}$.    It suffices to prove that $\lambda_1=\pm 1$ or $\lambda_2=\pm 1$.   Since $\det(I-\rho_\ell(\sigma)T)$ is a square by Lemma~\ref{L:square}, $\lambda_1\lambda_2^{-1}$ equals $\lambda_1\lambda_2$, $\lambda_1^{-1}\lambda_2^{-1}$ or $\lambda_1^{-1}\lambda_2$.   If $\lambda_1\lambda_2^{-1}=\lambda_1\lambda_2$, then $\lambda_2=\pm 1$.   If $\lambda_1\lambda_2^{-1}=\lambda_1^{-1}\lambda_2^{-1}$, then $\lambda_1=\pm 1$.      Finally, suppose that $\lambda_1\lambda_2^{-1}$ equals $\lambda_1^{-1} \lambda_2=(\lambda_1 \lambda_2^{-1})^{-1}$ and hence $\lambda_1=\varepsilon \lambda_2$ for some $\varepsilon\in \{\pm 1\}$.    The roots of $\det(I-gT)$ are thus $\varepsilon \lambda_2^2$, $\varepsilon$, $\varepsilon$ and $\varepsilon \lambda_2^{-2}$.    Since $\det(I-\rho_\ell(\sigma)T)$ is a square, we have $\varepsilon \lambda_2^2 = \varepsilon \lambda_2^{-2}$ and hence $\lambda_2^4=1$.  So if $\lambda_2\neq \pm 1$, then $\lambda_2^2=-1$ and hence $\det(I-gT)=(1-T)^2(1+T)^2=(1-T^2)^2$; however, this is impossible by Lemma~\ref{L:square}.  This proves the claim.

If the image of $B$ in $\PSL_2(\FF_\ell)$ has order $1$ or $\ell$, then $\lambda_2$ equals some $\varepsilon\in \{\pm 1\}$ and hence $\det(I-\rho_\ell(\sigma)T)= (1- \varepsilon \lambda_1 T)^2 (1-\varepsilon \lambda_1^{-1} T)^2= (1 - \varepsilon(\lambda_1+\lambda_1^{-1})T + T^2 )^2 = (1-\varepsilon \tr(A) T +T^2)^2$.

To complete the proof of the lemma, it remains to show, after possibly swapping $\vartheta_\ell$ and $\vartheta_\ell'$, that $\vartheta_\ell'(\Gal_{\QQ(i)})$ has cardinality $1$ or $\ell$.   If this is not true, then there are elements $\sigma_1$ and $\sigma_2$ of $\Gal_{\QQ(i)}$ such that $\vartheta_\ell(\sigma_1)^\ell \neq 1$ and $\vartheta_\ell'(\sigma_2)^{\ell} \neq 1$.   The order of $\vartheta_\ell'(\sigma_1)$ and $\vartheta_\ell(\sigma_2)$ are $1$ or $\ell$ by our claim.   After replacing $\sigma_1$ and $\sigma_2$ by an $\ell$-th power, we may assume that $\vartheta_\ell(\sigma_2)=1$ and $\vartheta_\ell'(\sigma_1)=1$.    So with $\sigma:=\sigma_1 \sigma_2$ we have $\vartheta_\ell(\sigma)^\ell \neq 1$ and $\vartheta_\ell'(\sigma)^\ell \neq 1$.
This is a contradiction since $\vartheta_\ell(\sigma)^\ell=1$ or $\vartheta_\ell'(\sigma)^\ell=1$ by our claim. 
\end{proof}

From now on, we take $\vartheta_\ell\colon \Gal_\QQ \to \PSL_2(\FF_\ell)$ and $\vartheta_\ell'\colon \Gal_\QQ \to \PSL_2(\FF_\ell)$ as in Lemma~\ref{L:need to switch}.

\begin{lemma} \label{L:other coset}
For $\sigma\in \Gal_\QQ-\Gal_{\QQ(i)}$, we have $\tr(\vartheta_\ell(\sigma))= \pm b$ for some $b\in \FF_\ell$ which satisfies $\det(I-\rho_\ell(\sigma)T)=1+(b^2-2)T^2+T^4$.
\end{lemma}
\begin{proof}
Fix $\sigma \in \Gal_\QQ - \Gal_{\QQ(i)}$.  Choose $A,B \in \SL_2(\FF_\ell)$ such that $\psi_\ell(\rho_\ell(\sigma)) = \widetilde\rho_\ell(\sigma)$ equals $A\otimes B$.  The image of $A$ and $B$ in $\PSL_2(\FF_\ell)$ is $\vartheta_\ell(\sigma)$ and $\vartheta_\ell'(\sigma)$, respectively.   
Take $\lambda_1, \lambda_2\in \FFbar_\ell^\times$ such that the eigenvalues of $A$ and $B$ are $\{\lambda_1,\lambda_1^{-1}\}$ and $\{\lambda_2,\lambda_2^{-1}\}$, respectively.  The roots of $\det(I-\rho_\ell(\sigma)T)$ in $\FFbar_\ell$ are thus $\lambda_1 \lambda_2$, $\lambda_1\lambda_2^{-1}$, $\lambda_1^{-1}\lambda_2$ and $\lambda_1^{-1}\lambda_2^{-1}$.  

By Lemma~\ref{L:need to switch} and our choice of $\vartheta_\ell'$, the coset of $B^2$ in $\PSL_2(\FF_\ell)$ has order $1$ or $\ell$.   Therefore, $\lambda_2^4=1$.   Since the group $\vartheta_\ell'(\Gal_{\QQ(i)})$ has order $1$ or $\ell$ by Lemma~\ref{L:need to switch}, we find that $\lambda_2$, up to a sign, does not depend on the initial choice of $\sigma\in \Gal_\QQ-\Gal_{\QQ(i)}$.  

Suppose that $\lambda_2=\pm 1$, and hence $\det(I - \rho_\ell(\sigma)T)$ equals $\det(I-AT)^2$ or $\det(I+AT)^2$ for each $\sigma \in \Gal_\QQ - \Gal_{\QQ(i)}$.   Since $3$ is inert in $\QQ(i)$, the polynomial $P_3(T)\equiv \det(I-\rho_\ell(\Frob_3)T) \pmod{\ell}$ must be a square.   The discriminant of $P_3(T)$, by Lemma~\ref{L:explicit}, is $2^{16}5^2/3^8$.  Since $\ell\geq 11$, we find that $P_3(T)\pmod{\ell}$ is separable which contradicts that it is a square.  

Therefore, $\lambda_2$ is a primitive $4$-th root of unity and hence
\begin{align*}
\det(I-(A\otimes B) T) &=(1-\lambda_1 \lambda_2 T) (1-\lambda_1\lambda_2^{-1} T)(1-\lambda_1^{-1} \lambda_2 T)(1-\lambda_1^{-1} \lambda_2^{-1} T)\\ 
&= (1- \lambda_1(\lambda_2+\lambda_2^{-1})T+\lambda_1^2 T^2)(1- \lambda_1^{-1}(\lambda_2+\lambda_2^{-1})T+\lambda_1^{-2} T^2)\\
&= (1+ \lambda_1^2 T^2)(1+ \lambda_1^{-2} T^2)\\
&= 1+(\lambda_1^2+\lambda_1^{-2})T^2 + T^4=1+((\tr A)^2-2)T^2+T^4. 
\end{align*}
The lemma is now immediate.
\end{proof}

\subsection{$L$-functions with $p\equiv 3 \pmod{4}$} \label{SS:L with p eq 3 mod 4}

We now show that the polynomial $P_p(T)$ with $p\equiv 3\pmod{4}$ are of a special form; we considered the primes $p\equiv 1\pmod{4}$ in \S\ref{SS:L with p eq 1 mod 4}.

\begin{proposition}  \label{P:P form mod 3}
For each prime $p\equiv 3 \pmod{4}$, we have $P_p(T)=1+(b^2-2) T^2 +T^4$ for a unique non-negative $b\in \ZZ[1/p]$.
\end{proposition}
\begin{proof}
Fix a prime $p\equiv 3\pmod{4}$.     By Lemma~\ref{L:other coset}, for each prime $\ell\geq 11$ with $\ell\neq p$ we have
\[
P_p(T) \equiv \det(I - \rho_\ell(\Frob_p) T) = 1+(b^2-2) T^2 +T^4 \pmod{\ell}
\]
for some $b\in \FF_\ell$.    Since $P_p(T)$ modulo $\ell$ is of the form $ 1+(b^2-2) T^2 +T^4$ modulo $\ell$ for all but finitely many primes $\ell$, we deduce that $P_p(T)$ is of the form $ 1+(b^2-2) T^2 +T^4$ for some $b\in \QQ$.     Therefore, $L(T,E_p)= P_p(pT)= 1+(p^2b^2-2p^2) T^2 +p^4T^4$.  Since $L(T,E_p)$ has integer coefficients, unique factorization shows that $b^2$ is uniquely determined and that $bp \in \ZZ$.   Uniqueness for $b$ is obtained by imposing the condition $b\geq 0$.
\end{proof}

\section{Ramification at $2$} \label{S:ramification at 2}

Fix an inertia subgroup $\calI_2$ of $\Gal_\QQ$ at the prime $2$.     The goal of this section is to prove the following.

\begin{proposition} \label{P:explicit unipotent}
For any prime $\ell\geq 11$ and $g\in \calI_2$, $\rho_{\ell}(g)^{12}$ is unipotent.
\end{proposition}

Take any $\ell\geq 11$.  Let 
\[
\varphi_\ell \colon \Gal_\QQ \to \Aut_{\QQ_\ell}(H^2_{\et}(X_\Qbar,\QQ_\ell))
\]
be the representation describing the Galois action on $H^2_{\et}(X_\Qbar,\QQ_\ell)$.   Grothendieck proved that there is an open subgroup $\calI'$ of $\calI_2$ such that $\varphi_\ell(g)$ is unipotent for all $g\in \calI'$; see the appendix of \cite{MR0236190}.   Thus for each $g\in \calI_2$, some positive power of $\varphi_\ell(g)$ is unipotent.   For each $g\in \calI_2$, let $m_g$ be the smallest positive integer for which $\varphi_\ell(g)^{m_g}$ is unipotent.  

\begin{lemma}
The integer $m_g$ does not depend on $\ell$.
\end{lemma}
\begin{proof}
Take any $g\in \calI_2$.   It suffices to show that $\varphi_\ell(g)$ is unipotent for one prime $\ell$ if and only if it is unipotent for all $\ell$.  Let $d$ be the dimension of $H^2_{\et}(X_\Qbar,\QQ_\ell)$ over $\QQ_\ell$; it does not depend on $\ell$.

Define $t_g:=\tr(\varphi_\ell(g))$.   Since some power of $\varphi_\ell(g)$ is unipotent, we find that the eigenvalues of $\varphi_\ell(g)$ are roots of unity.   Therefore, $\varphi_\ell(g)$ is unipotent if and only if $t_g=d$.    It thus suffices to show that $t_g$ is an integer that does not depend on $\ell$.   This follows from Corollary~2.5 of \cite{MR1715253} and uses that $X$ is a smooth proper surface.  
\end{proof}

Let $\phi_\ell$ be the Galois representation of \S\ref{SS:intro surface}.  For a prime $\ell$ and $g\in \calI_2$, let $n_{g,\ell}$ be the smallest positive integer for which $\phi_{\ell}(g)^{n_{g,\ell}}$ is unipotent.  

\begin{lemma}  \label{L:mg prop}
Take any prime $\ell\geq 11$ and  $g\in \calI_2$.   
\begin{romanenum}
\item \label{L:mg prop divides}
The integer $n_{g,\ell}$ divides $m_g$.
\item \label{L:mg prop eq}
If $\ell\nmid m_g$, then $m_g = n_{g,\ell}$.
\end{romanenum}
\end{lemma}
\begin{proof}
By Lemma~\ref{L:free over Zl}, we can identify $\varphi_\ell$ with the representation $\Gal_\QQ \to \Aut_{\ZZ_\ell}(\Lambda)$ describing the Galois action on $\Lambda:=H^2_{\et}(X_\Qbar,\ZZ_\ell)$.   Again by Lemma~\ref{L:free over Zl}, the quotient $\Lambda/\ell\Lambda$ is isomorphic to $H^2_\et(X_\Qbar,\FF_\ell)$ and the action of $\Gal_\QQ$ on the quotient gives rise to $\phi_\ell$.  Since $\varphi_\ell(g)^{m_g}$ is unipotent, we deduce that $\phi_\ell(g)^{m_g}$ is unipotent.  Therefore, $n_{g,\ell}$ divides $m_g$ and $m_g/n_{g,\ell}$ is a power of $\ell$.  This proves (\ref{L:mg prop divides}) and we have $m_g=n_{g,\ell}$ when $\ell\nmid m_g$.
\end{proof}

\begin{proof}[Proof of Proposition~\ref{P:explicit unipotent}]
Fix $g\in \calI_2$.    Take any prime $\ell\geq 11$.    By Lemma~\ref{L:X connection}, $n_{g,\ell}$ is also the smallest positive integer for which $\rho_{\ell}(g)^{n_{g,\ell}}$ is unipotent.   By Lemma~\ref{L:mg prop}(\ref{L:mg prop divides}), it suffices to prove that $m_g$ divides $12$.  

Now take any prime $\ell\geq 11$ which does not divide $m_g$.  By Lemma~\ref{L:mg prop}(\ref{L:mg prop eq}), $m_g$ is the smallest positive integer for which $\rho_\ell(g)^{m_g}$ is unipotent.

The order of $\vartheta_\ell'(g)$ divides $2$.   The order of any element of $\PSL_2(\FF_\ell)$, and in particular $\vartheta_\ell(g)$, divides $\ell$, $(\ell-1)/2$ or $(\ell+1)/2$.   Therefore, the order of the image of $\rho_\ell(g)$ in $\Omega(V_\ell)/\{\pm I\}$ divides $2\ell$, $\lcm(2,(\ell-1)/2)$ or $\lcm(2,(\ell+1)/2)$.    The order $e_{g,\ell}$ of $\rho_\ell(g)$ thus divides $4\ell$, $\lcm(4,\ell-1)$ or $\lcm(4,\ell+1)$.     

Since $m_g$ divides $e_{g,\ell}$ and is not divisible by $\ell$, we deduce that $m_g$ divides $\lcm(4,\ell-1)$ or $\lcm(4,\ell+1)$ for all sufficiently large primes $\ell$.   Using Dirichlet's theorem on arithmetic progressions, we can then deduce from this that $m_g$ divides $12$.  
\end{proof}

\section{Proof of Theorems~\ref{T:main}, \ref{T:main 2} and \ref{T:main 3}} \label{S:main proof}

Fix a prime $\ell\geq 11$.   Let $\calI$ and $\calI_2$ be inertia subgroups of $\Gal_\QQ$ corresponding to the primes $\ell$ and $2$, respectively.

 Let $\rho_\ell\colon \Gal_\QQ \to \Or(V_\ell)$ be the representation of \S\ref{SS:intro rep}; its image is contained in $\Omega(V_\ell)$ by Lemma~\ref{L:main containment}.   Let $\vartheta_\ell$ and $\vartheta_\ell'$ be the homomorphisms $\Gal_\QQ \to \PSL_2(\FF_\ell)$ from \S\ref{SS:psi} chosen so that they satisfy the conclusion of Lemma~\ref{L:need to switch}.   
 
 We shall prove that $\vartheta_\ell$ is surjective.   To do this, we first describe the maximal subgroups of $\PSL_2(\FF_\ell)$.  The description of subgroups of $\GL_2(\FF_\ell)$ from \S2.4 and \S2.6 of \cite{MR0387283} shows that if $M$ is a maximal subgroup of $\PSL_2(\FF_\ell)$, then one of the following holds:
\begin{itemize}
\item $M$ is a Borel subgroup,
\item $M$ is the normalizer of a Cartan subgroup,
\item $M$ is isomorphic to $\mathfrak{A}_4$, $\mathfrak{S}_4$ or $\mathfrak{A}_5$.
\end{itemize}

A Borel subgroup of $\PSL_2(\FF_\ell)$ is a group whose inverse image in $\SL_2(\FF_\ell)$ is conjugate to the subgroup of upper triangular matrices.  

A Cartan subgroup $C$ of $\PSL_2(\FF_\ell)$ is a maximal cyclic subgroup whose order is relatively prime to $\ell$.  The group $C$ is cyclic of order $(\ell-1)/2$ or $(\ell+1)/2$; we say that $C$ is \defi{split} or \defi{non-split}, respectively.    Let $N$ be the normalizer of $C$ in $\PSL_2(\FF_\ell)$.   The group $C$ has index $2$ in $N$ and one can show that $\tr(A)=0$ for all $A\in N-C$.

\subsection{Borel case}  \label{SS:Borel}

Assume that $\vartheta_\ell(\Gal_\QQ)$ is contained in a Borel subgroup $\bbar{B}$ of $\PSL_2(\FF_\ell)$.  Let $B$ be the inverse image of $\bbar{B}$ under the quotient homomorphism $\SL_2(\FF_\ell)\to \PSL_2(\FF_\ell)$.  There is a non-zero vector $v\in \FF_\ell^2$ such that the subspace $\FF_\ell\cdot  v$ is stable under the action of $B$.  Let $\varphi\colon B \to \FF_\ell^\times$ be the homomorphism such that $Av=\varphi(A) v$ for all $A\in B$; it gives rise to a character $\bbar{\varphi} \colon \bbar{B} \to \FF_\ell^\times/\{\pm 1\}$.   Let $\alpha \colon \Gal_\QQ \to \FF_\ell^\times/\{\pm 1\}$ be the character $\bbar{\varphi} \circ \vartheta_\ell$.

\begin{lemma} \label{L:root lift}
For each $\sigma\in \Gal_{\QQ(i)}$, there is a root $a\in \FF_\ell^\times$ of $\det(I-\rho_\ell(\sigma) T)$ whose image in $\FF_\ell^\times/\{\pm 1\}$ is $\alpha(\sigma)$.
\end{lemma}
\begin{proof}
Choose a matrix $A\in \SL_2(\FF_\ell)$ whose image in $\PSL_2(\FF_\ell)$ is $\vartheta_\ell(\sigma)$.   By Lemma~\ref{L:need to switch}, after possibly replacing $A$ by $-A$, we have $\det(I-\rho_\ell(\sigma) T) = (1 -\tr(A) T + T^2)^2$.   Since $A$ belong to $B$, $\varphi(A)$ is a root of $\det(I-\rho_\ell(\sigma) T) = (1 -\tr(A) T + T^2)^2$.  So $\varphi(A)\in \FF_\ell^\times$ is a representative of $\bbar\varphi(\vartheta_\ell(\sigma))=\alpha(\sigma)$ and a root of $\det(I-\rho_\ell(\sigma) T)$.
\end{proof}

\begin{lemma} \label{L:gamma}
There is an integer $e\in \{-1,0,1\}$ such that the character $\gamma \colon \Gal_\QQ \to \FF_\ell^\times/\{\pm 1\}$ given by $\sigma\mapsto \alpha(\sigma) \cdot \chi_\ell(\sigma)^{-e}$ is unramified at all odd primes.
\end{lemma}
\begin{proof}
Fix a $\tau\in \calI$ whose image topologically generates $\calI_t$.     By Lemma~\ref{L:root lift}, there is a representative $a\in \FF_\ell^\times$ of $\alpha(\tau)$ which is a root of $\det(I-\rho_\ell(\tau)T)$.  So there is a one-dimensional subspace $\calW$ of $V_\ell$ which is stable under the action of $\calI$, and $\tau$ acts on $\calW$ as multiplication by $a$.   By Lemma~\ref{L:V semisimplification for inertia}(\ref{I:V semisimplification for inertia d}), $a=\chi_\ell(\tau)^e$ for some $e\in \{-1,0,1\}$.   Define the character $\gamma\colon \Gal_\QQ \to \FF_\ell^\times/\{\pm 1\}$ by $\sigma\mapsto \alpha(\sigma) \chi_\ell(\sigma)^{-e}$.    We have $\gamma(\tau)=1$, so $\gamma(\calI)=\gamma(\calI_t)=1$.   Therefore, $\gamma$ is unramified at $\ell$.  The character $\gamma$ is unramified at the primes $p\nmid 2\ell$ since $\rho_\ell$ and $\chi_\ell$ are unramified at such primes.
\end{proof}

\begin{lemma} \label{L:borel contradictions}
For each prime $p\nmid 2\ell$, we have $P_p^{(4)}(\epsilon p^{4e})\equiv 0 \pmod{\ell}$ for some $\epsilon \in \{\pm 1\}$ and $e\in\{0,1\}$.  
\end{lemma}
\begin{proof}
 Take $e$ and $\gamma$ as in Lemma~\ref{L:gamma}.   Since the image of $\gamma$ lies in an abelian group, we find that $\gamma(\calI_2)$ does not depend on the choice of $\calI_2$.  
 We claim that $\gamma(\calI_2)=\gamma(\Gal_\QQ)$.  If $\gamma(\calI_2)$ is a proper subgroup of $\gamma(\Gal_\QQ)$, then it gives rise to a non-trivial extension of $\QQ$ that is unramified at all primes.  The claim follows since $\QQ$ has no such extension.

The group $(\FF_\ell^\times)^2/\{\pm 1\}$ is cyclic, so $\gamma(\Gal_\QQ)=\gamma(\calI_2)$ is cyclic of some order $m$.   Proposition~\ref{P:explicit unipotent} implies that $m$ divides $12$.    We claim that the cardinality of $\gamma(\Gal_\QQ)$ divides $4$.   If $3$ divides $|\gamma(\Gal_\QQ)|$, then $\gamma$ gives rise to a cubic abelian extension of $\QQ$ that is unramified outside of $2$.  However, by class field theory no such cubic extensions exist, so the claim follows.

Take any $\sigma\in \Gal_\QQ$.   We have $\gamma^4=1$, so $\chi_\ell(\sigma)^{4e}$ is a representative of $\alpha(\sigma)^4= \alpha(\sigma^4)$.  By Lemma~\ref{L:root lift}, $\epsilon \chi_\ell(\sigma)^{4e}$ is a root of $\det(I-\rho_\ell(\sigma)^4 T)$ for some $\epsilon\in \{\pm 1\}$.     Since $\rho_\ell(\sigma) \in \SO(V_\ell)$ by Lemma~\ref{L:SO image}, we find that $\chi_\ell(\sigma)^{-4e}$ is also a root of $\det(I-\rho_\ell(\sigma)^4 T)$.

 Now take any prime $p\nmid 2\ell$.   We have $\chi_\ell(\Frob_p)\equiv p \pmod{\ell}$ and $\det(I-\rho_\ell(\Frob_p)^4 T) \equiv P_p^{(4)}(T) \pmod{\ell}$.   Therefore, $P_p^{(4)}(\epsilon p^{4e'}) \equiv 0 \pmod{\ell}$ where $e'=|e|$.
\end{proof}

Using (\ref{E:P4 explicit}), we find that 
$P_3^{(4)}(1)= 2^{12} 5^2/3^8$, $P_3^{(4)}(-1)= 2^4/3^8$, $P_3^{(4)}(3^4)= 2^{12} 3^2 5^2 7^2$,  $P_3^{(4)}(-3^4)= 2^4 1601^2$,
$P_5^{(4)}(1)= 2^{14} 3^2/5^8$, $P_5^{(4)}(-1)= 2^4 23^4/5^8$, $P_5^{(4)}(5^4)= 2^{14} 3^2 5^2 7^2 29^2$ and  $P_5^{(4)}(-5^4)= 2^4 97^2 1009^2$.   By Lemma~\ref{L:borel contradictions} with $p\in\{3,5\}$ and our ongoing assumption $\ell\geq 11$, we obtain a contradiction.

Therefore, $\vartheta_\ell(\Gal_\QQ)$ is not contained in a Borel subgroup of $\PSL_2(\FF_\ell)$.

\subsection{Cartan case} \label{SS:Cartan}

We now suppose that $\vartheta_\ell(\Gal_\QQ)$ is contained in the normalizer of a Cartan subgroup $C$ of $\PSL_2(\FF_\ell)$.  

\begin{lemma} \label{L:no normalizers}
We have $\vartheta_\ell(\Gal_\QQ) \subseteq C$ and the group $\vartheta_\ell(\calI)$ is either $1$ or $C$.  
\end{lemma}
\begin{proof}
Let $N$ be the normalizer of $C$ in $\PSL_2(\FF_\ell)$.  By Proposition~\ref{P:ramification at ell}(\ref{P:ramification at ell c}), the group $\rho_\ell(\calI)/\rho_\ell(\calP)$ is cyclic of order $1$, $\ell-1$ or $\ell+1$.   If $\rho_\ell(\calI)/\rho_\ell(\calP)=1$, then $\vartheta_\ell(\calI)=1$ since $\ell \nmid |N|$.  

Now assume that $\rho_\ell(\calI)/\rho_\ell(\calP)$ is cyclic of order $\ell-1$ or $\ell+1$.   The image of $\rho_\ell(\calI)$ in $\Omega(V_\ell)/\{\pm I\}$ thus contains a cyclic group of order $(\ell-1)/2$ or $(\ell+1)/2$.   The group $\vartheta_\ell'(\calI)$ is of order $1$ or $\ell$ by Lemma~\ref{L:need to switch} since $\calI \subseteq \Gal_{\QQ(i)}$.   Since $\ell\nmid |N|$, we deduce that $\vartheta_\ell(\calI)$ is a cyclic group containing a subgroup of order $(\ell-1)/2$ or $(\ell+1)/2$.  This implies that $\vartheta_\ell(\calI)$ is a Cartan subgroup of $\PSL_2(\FF_\ell)$.  The group $N$ contains a unique Cartan subgroup, so $\vartheta_\ell(\calI)=C$.

It remains to show that $\vartheta_\ell(\Gal_\QQ) \subseteq C$.  Let $\varepsilon\colon \Gal_\QQ \to \{\pm 1\}$ be the character obtained by composing $\vartheta_\ell\colon \Gal_\QQ \to N$ with the quotient map  $N\to N/C \cong \{\pm 1\}$.  It thus suffices to show that $\varepsilon=1$.    
  
Suppose $p\nmid 2\ell$ is a prime that satisfies $\varepsilon(\Frob_p) = -1$ and hence $\vartheta_\ell(\Frob_p) \in N- C$. Recall that every $g\in N-C$ satisfies $\tr(g)=0$.  By Lemmas~\ref{L:need to switch} and \ref{L:other coset}, the polynomial $\det(I- \rho_\ell(\Frob_p)T) \equiv P_p(T) \pmod{\ell}$ is either $(1+T^2)^2$ or $1-2T^2+T^4=(1-T^2)^2$.   Using the values of $P_3(T)$ and $P_5(T)$ from Lemma~\ref{L:explicit}, this shows that $\varepsilon(\Frob_3)=1$ and $\varepsilon(\Frob_5)=1$.   

The character $\varepsilon$ is unramified at $p\nmid 2\ell$ since $\rho_\ell$ is unramified at such primes.  The character $\varepsilon$ is also unramified at $\ell$ since $\vartheta_\ell(\calI)\subseteq C$.   Let $K$ be the fixed field in $\Qbar$ of $\ker(\varepsilon)$.   The extension $K/\QQ$ is unramified at all odd primes and has degree at most $2$, so $K$ is $\QQ$, $\QQ(i)$, $\QQ(\sqrt{2})$ or $\QQ(\sqrt{-2})$.   The primes $3$ and $5$ split in $K$, which rules out $\QQ(i)$, $\QQ(\sqrt{2})$ and $\QQ(\sqrt{-2})$.   Therefore, $K=\QQ$ and hence $\varepsilon=1$.  
\end{proof}

A split Cartan subgroup of $\PSL_2(\FF_\ell)$ lies in a Borel subgroup.   So by the case considered in \S\ref{SS:Borel} and Lemma~\ref{L:no normalizers}, we deduce that $C$ is non-split.

\begin{lemma}  \label{L:non-split ramified}
The representation $\vartheta_\ell$ is unramified at $\ell$.
\end{lemma}
\begin{proof}
Suppose that $\vartheta_\ell$ is ramified at $\ell$.   By Lemma~\ref{L:no normalizers}, we have $\vartheta_\ell(\Gal_\QQ)=\vartheta_\ell(\calI)=C$. So $\vartheta_\ell$ gives rise to an abelian extension $K/\QQ$ of degree $(\ell+1)/2$ that is totally ramified at $\ell$.   There is thus a totally ramified abelian extension $K'/\QQ_\ell$ of degree $(\ell+1)/2$.   By local class field theory, $\Gal(K'/\QQ_\ell)$ must be a quotient of $\ZZ_\ell^\times\cong \FF_\ell^\times \times \ZZ_\ell$.   However, this is impossible since $(\ell+1)/2>1$ is relatively to the integers $(\ell-1)\ell^e$.  
\end{proof}

\begin{lemma}  \label{L:non-split contradiction}
For each prime $p\nmid 2\ell$, the polynomial $P_p^{(4)}(T)$ is congruent modulo $\ell$ to $(1-T)^4$ or $(1+T)^4$.
\end{lemma}
\begin{proof}
Since the image of $\vartheta_\ell$ lies in the cyclic group $C$, we find that $\vartheta_\ell(\calI_2)$ does not depend on the choice of $\calI_2$.   We claim that $\vartheta_\ell(\calI_2)=\vartheta_\ell(\Gal_\QQ)$.  If $\vartheta_\ell(\calI_2)$ is a proper subgroup of $\vartheta_\ell(\Gal_\QQ)$, then it gives rise to a non-trivial extension of $\QQ$ that is unramified at all primes (this uses Lemma~\ref{L:non-split ramified}).  The claim follows since $\QQ$ has no such extension.

The group $C$ is cyclic, so $\vartheta_\ell(\Gal_\QQ)=\vartheta_\ell(\calI_2)$ is cyclic of some order $m$.   Proposition~\ref{P:explicit unipotent} implies that $m$ divides $12$.    We claim that the cardinality of $\vartheta_\ell(\Gal_\QQ)$ divides $4$.   If $3$ divides $|\vartheta_\ell(\Gal_\QQ)|$, then $\vartheta_\ell$ gives rise to a cubic abelian extension of $\QQ$ that is unramified outside of $2$.  However, by class field theory no such cubic extension exist, so the claim follows.

Take any $\sigma\in \Gal_\QQ$.   We have $\vartheta_\ell(\sigma^4)=\vartheta_\ell(\sigma)^4=I$, so $\det(I-\rho_\ell(\sigma)^4 T)$ is $(1+2T+T^2)^2=(1+T)^4$ or $(1-2T+T^2)^2=(1-T)^4$ by Lemma~\ref{L:need to switch}.   The lemma follows since $P_p^{(4)}(T)\equiv \det(I-\rho_\ell(\Frob_p)^4T)  \pmod{\ell}$ for every prime $p\nmid 2\ell$.  
\end{proof}

From (\ref{E:P4 explicit}), we have $P_3^{(4)}(1)= 2^{12} 5^2/3^8$ and $P_3^{(4)}(-1)= 2^4/3^8$.    Since $\ell\geq  11$, we have $P_3^{(4)}(1)\not\equiv 0 \pmod{\ell}$ and $P_3^{(4)}(-1)\not\equiv 0 \pmod{\ell}$.   However, this contradicts Lemma~\ref{L:non-split contradiction} with $p=3$.    Therefore, $\vartheta_\ell(\Gal_\QQ)$ is not contained in the normalizer of a Cartan subgroup of $\PSL_2(\FF_\ell)$.

\subsection{Exceptional case} \label{SS:exceptional}

Assume that $\vartheta_\ell(\Gal_\QQ)$ is contained in a subgroup $M$ of $\PSL_2(\FF_\ell)$ that is isomorphic to $\mathfrak{A}_4$, $\mathfrak{S}_4$ or $\mathfrak{A}_5$.  As observed in \cite{MR0387283}*{\S2.6}, for every $g\in M$, $u:=\tr(g)^2 \in \FF_\ell^\times$ is an element of  $\{0,1,2,4\}$ or satisfies $u^2-3u+1=0$.    For each prime $p\nmid 2\ell$, define $u_p:= \tr(\vartheta_\ell(\Frob_p))^2\in \FF_\ell$.   We thus have $u_p\in \{0,1,2,4\}$ or $u_p^2-3u_p+1=0$. 

By Lemmas~\ref{L:explicit} and \ref{L:need to switch}, we have $u_3 = 16/9$.   So $u_3 = 2^4/3^2$, $u_3-1 = 7/3^2$, $u_3 -2 = -2/3^2$, $u_3-4 = -2^2 5/3^2$ and $u_3^2-3u_3+1 = - 5\cdot 19/3^4$.   Since $\ell \geq 11$, the prime $\ell$ must be $19$. 

By Lemmas~\ref{L:explicit} and \ref{L:other coset}, we have $u_5= 4/25$.  So $u_5= 2^2/5^2$, $u_5-1 = -3 \cdot 7 /5^2$, $u_5-2 = -2 \cdot 23/ 5^2$, $u_5 -4 = -2^5 3/5^2$ and $u_5^2-3u_5+1 = 11\cdot 31/5^4$.  However, since $\ell=19$, this contradicts that $u_5\in \{0,1,2,4\}$ or $u_5^2-3u_5+1=0$.   

Therefore, $\vartheta_\ell(\Gal_\QQ)$ is not contained in a subgroup of $\PSL_2(\FF_\ell)$ which is isomorphic to $\mathfrak{A}_4$, $\mathfrak{S}_4$ or $\mathfrak{A}_5$.

\subsection{Proof of Theorem~\ref{T:main 2}}

Fix a prime $\ell\geq 11$.   Since $\vartheta_\ell(\Gal_\QQ)$ is a quotient of $\rho_\ell(\Gal_\QQ)$, it thus suffices to prove that $\vartheta_\ell(\Gal_\QQ)=\PSL_2(\FF_\ell)$.

\begin{lemma} \label{L:surjectivity 1}
The representation $\vartheta_\ell \colon \Gal_\QQ \to \PSL_2(\FF_\ell)$ is surjective.
\end{lemma}
\begin{proof}
If $\vartheta_\ell$ is not surjective, then its image lies in a maximal subgroup $M$ of $\PSL_2(\FF_\ell)$.    From \S\S\ref{SS:Borel}--\ref{SS:exceptional}, we find that $M$ is not a Borel subgroup, not the normalizer of a Cartan subgroup, and not isomorphic to $\mathfrak{A}_4$, $\mathfrak{S}_4$ and $\mathfrak{A}_5$.    This contradicts the classification of maximal subgroups of $\PSL_2(\FF_\ell)$ described at the beginning of \S\ref{S:main proof}.  Therefore, $\vartheta_\ell$ is surjective.
\end{proof}

\subsection{Proof of Theorem~\ref{T:main}}

The theorem is an immediate consequence of Theorem~\ref{T:main 2} for primes $p\geq 11$.  The groups $\PSL_2(\FF_5)$ and $\PSL_2(\FF_7)$ are both known to occur as the Galois group of an extension of $\QQ$; for example, this follows from the results of Shih mentioned in the introduction (or more concretely one can just write down polynomials with these Galois groups).   

\begin{remark}
We can actually show that for each prime $\ell\geq 5$, there is a Galois extension $K/\QQ$ which is unramified away from $2$ and $\ell$ such that $\Gal(K/\QQ)\cong \PSL_2(\FF_\ell)$.   For $\ell \geq 11$, this is clear from Theorem~\ref{T:main 2} since $\rho_\ell$ is unramified away from $2$ and $\ell$.   One can show that the polynomial $x^5+20x-16$ has discriminant $2^{16} 5^6$ and Galois group isomorphic to $\PSL_2(\FF_5)$.   
One can show that the polynomial $x^7-7x^5-14x^4-7x^3-7x+2$ has discriminant $2^{20} 7^8$ and Galois group isomorphic to $\PSL_2(\FF_7)$.   
\end{remark}


\subsection{Proof of Theorem~\ref{T:main 3}}
Theorem~\ref{T:main 3} is an easy consequence of Theorem~\ref{T:main 2} and Lemma~\ref{L:X connection}.  

\section{The image of $\rho_\ell$} \label{S:image}

Let $\calH_\ell$ be the subgroup of $\SL_4(\FF_\ell)$ consisting of the matrices $\left(\begin{smallmatrix}A & 0 \\  0 & A\end{smallmatrix}\right)$ with $A\in \SL_2(\FF_\ell)$.   Let $\calG_\ell$ be the subgroup of $\SL_4(\FF_\ell)$ generated by $\calH_\ell$ and the matrix $\gamma:=\left(\begin{smallmatrix} 0 & -I \\  I & 0 \end{smallmatrix}\right)$.  We have $\gamma^2 = -I$ and $\calH_\ell$ commutes with $\gamma$, so $\calH_\ell$ is a normal subgroup of $\calG_\ell$ with index $2$.  

The following describes the groups $\rho_\ell(\Gal_\QQ)$ and $\rho_\ell(\Gal_{\QQ(i)})$ up to conjugation in $\Aut_{\FF_\ell}(V_\ell)$.

\begin{theorem} \label{T:final}
Take any prime $\ell\geq 11$.  There is a representation $\varrho_\ell \colon \Gal_\QQ \to \GL_4(\FF_\ell)$ isomorphic to $\rho_\ell$ such that $\varrho_\ell(\Gal_{\QQ(i)}) = \calH_\ell$ and $\varrho_\ell(\Gal_\QQ) = \calG_\ell$.
\end{theorem}

Theorem~\ref{T:final} is of course a generalization of Theorem~\ref{T:main 2}; note that $\calG_\ell/\langle \gamma \rangle \cong \PSL_2(\FF_\ell)$.

\subsection{Proof of Theorem~\ref{T:final}}
Fix a prime $\ell\geq 11$ and define $\calV_\ell := \FF_\ell^2 \otimes_{\FF_\ell} \FF_\ell^2$.  Let
\[
\widetilde\rho_\ell \colon \Gal_\QQ \to \SL_2(\FF_\ell) \otimes \SL_2(\FF_\ell) \subseteq \Aut_{\FF_\ell}(\calV_\ell)
\]
and $\vartheta_\ell , \vartheta_\ell'\colon \Gal_\QQ \to \PSL_2(\FF_\ell)$ be the representations of \S\ref{SS:psi}.    We choose $\vartheta_\ell$ and $\vartheta_\ell'$ so that they satisfy the conditions of Lemma~\ref{L:need to switch}.

\begin{lemma} \label{L:beta image}
The group $\rho_\ell(\Gal_{\QQ(i)})$, and hence also $\widetilde\rho_\ell(\Gal_{\QQ(i)})$, is isomorphic to $\SL_2(\FF_\ell)$. 
\end{lemma}
\begin{proof}
Let $W$ be an irreducible $\FF_\ell[\Gal_{\QQ(i)}]$-submodule of $V_\ell$.  If $W$ has dimension $1$ over $\FF_\ell$, then Proposition~\ref{P:V semisimplification}(\ref{I:V semisimplification a}) tells us that $\rho_\ell(\Gal_{\QQ(i)})$ is a solvable group.   This is impossible since the non-abelian simple group $\PSL_2(\FF_\ell)$ is a quotient of $\rho_\ell(\Gal_\QQ)$, and hence also of $\rho_\ell(\Gal_{\QQ(i)})$, by Theorem~\ref{T:main 2}.    By Proposition~\ref{P:V semisimplification}, we deduce that $W$ has dimension $2$ over $\FF_\ell$ and that $V_\ell$ and $W\oplus W$ are isomorphic $\FF_\ell[\Gal_{\QQ(i)}]$-modules.  Let $\beta \colon \Gal_{\QQ(i)} \to \Aut_{\FF_\ell}(W)\cong \GL_2(\FF_\ell)$ be the representation describing the Galois action on $W$.   It thus suffices to show that $\beta(\Gal_{\QQ(i)})=\SL_2(\FF_\ell)$ since $\beta_\ell(\Gal_{\QQ(i)})\cong \rho_\ell(\Gal_{\QQ(i)})$ and since $\rho_\ell$ and $\widetilde\rho_\ell$ are isomorphic.

We have $\beta_\ell(\Gal_{\QQ(i)})\subseteq \SL_2(\FF_\ell)$ by Lemma~\ref{L:beta image eq1}.  The group $\PSL_2(\FF_\ell)$ is a quotient of $\rho_\ell(\Gal_{\QQ(i)})$ and hence is also a quotient of $\beta(\Gal_{\QQ(i)})$.  So $\beta(\Gal_{\QQ(i)})$ is a subgroup of $\SL_2(\FF_\ell)$ that has a quotient isomorphic $\PSL_2(\FF_\ell)$.   Since $\SL_2(\FF_\ell)$ is perfect (and in particular has no subgroups of index $2$), we deduce that $\beta(\Gal_{\QQ(i)})=\SL_2(\FF_\ell)$.  
\end{proof}

\begin{lemma} \label{L:theta trivial}
We have $\vartheta_\ell'(\Gal_{\QQ(i)})=1$.
\end{lemma}
\begin{proof}
By Lemma~\ref{L:need to switch} and our choice of $\vartheta_\ell$, the group $\vartheta'_\ell(\Gal_{\QQ(i)})$ has order $1$ or $\ell$.   Since $\vartheta'_\ell(\Gal_{\QQ(i)})$ is a quotient of the perfect group $\rho_\ell(\Gal_{\QQ(i)})\cong \SL_2(\FF_\ell)$ by Lemma~\ref{L:beta image}, we deduce that has $\vartheta'_\ell(\Gal_{\QQ(i)})=1$.
\end{proof}

Lemma~\ref{L:theta trivial} implies that $\widetilde\rho_\ell(\Gal_{\QQ(i)})$ is a subgroup of $\SL_2(\FF_\ell) \otimes \langle \pm I \rangle = \SL_2(\FF_\ell) \otimes \langle I \rangle$.   So from Lemma~\ref{L:beta image}, we deduce that 
\begin{equation}  \label{E:SL2 final}
\widetilde\rho_\ell(\Gal_{\QQ(i)})= \SL_2(\FF_\ell) \otimes \langle I \rangle.
\end{equation}

\begin{lemma} \label{L:matrix B}
There is a matrix $B\in \SL_2(\FF_\ell)$ such that $B^2=-I$ and $\vartheta_\ell'(\Gal_\QQ)$ is generated by the image of $B$ in $\PSL_2(\FF_\ell)$.
\end{lemma}
\begin{proof}
By Lemma~\ref{L:theta trivial}, there is a matrix $B\in \SL_2(\FF_\ell)$ such that the image $\bbar{B}$ of $B$ in $\PSL_2(\FF_\ell)$ generates $\vartheta_\ell'(\Gal_\QQ)$ and $\bbar{B}^2=-I$.   We have $\vartheta_\ell'(\sigma)=\bbar B$ for all $\sigma\in \Gal_\QQ-\Gal_{\QQ(i)}$.  The proof of Lemma~\ref{L:other coset} shows that the eigenvalues of $B$ are primitive $4$-th roots of unity, so $B^2=-I$.
\end{proof}

By Lemma~\ref{L:matrix B}, the group $\vartheta_\ell'(\Gal_\QQ)$ is generated by the image in $\PSL_2(\FF_\ell)$ of some matrix $B\in \SL_2(\FF_\ell)$ that satisfies $B^2 = - I$.   So there is an inclusion $\widetilde\rho_\ell(\Gal_\QQ) \subseteq \SL_2(\FF_\ell) \otimes \langle B \rangle$.   This inclusion with (\ref{E:SL2 final}) proves that 
\begin{equation} \label{E:SL2plus final}
\widetilde\rho_\ell(\Gal_\QQ) = \SL_2(\FF_\ell) \otimes \langle B \rangle.
\end{equation}
After conjugating $\widetilde\rho_\ell$ by an element of $\langle I \rangle \otimes \GL_2(\FF_\ell)$, we may assume that $B=\left(\begin{smallmatrix} 0 & -1 \\  1 & 0 \end{smallmatrix}\right)$.

Let $e_1$ and $e_2$ be the standard basis of $\FF_\ell^2$.   Define 
\[
\varrho_\ell \colon \Gal_\QQ \to \GL_4(\FF_\ell)
\]
to be the representation obtained from $\widetilde\rho_\ell$ by using the basis $\beta:=\{e_1\otimes e_1, e_2 \otimes e_1, e_1 \otimes e_2,  e_2 \otimes e_2 \}$ of $\calV_\ell$.   Using (\ref{E:SL2 final}), we find that $\varrho_\ell(\Gal_{\QQ(i)}) = \calH_\ell$.   With respect to the basis $\beta$, $I\otimes B$ acts on $\calV_\ell$ via the matrix $\gamma$.  So $\varrho_\ell(\Gal_\QQ)$ is generated by $\calH_\ell$ and $\gamma$, and hence equals $\calG_\ell$.

\subsection{Modularity}  \label{S:modularity}

Fix a prime $\ell \geq 11$ and take $\varrho_\ell$ as in Theorem~\ref{T:final}.  Define the ring $\OO:=\ZZ[i]$.  We view $\FF_\ell^4$ as an $\OO$-module by letting $i$ act as $\gamma$  (note that $\gamma^2=-I$).   This action turns $\FF_\ell^4$ into an $\OO/\ell\OO$-module that is free of rank $2$.   Let $\{e_1,e_2,e_3,e_4\}$ be the standard basis of $\FF_\ell^4$.  One can verify that $\beta:=\{e_1,e_2\}$ is a basis of $\FF_\ell^4$ as an $\OO/\ell \OO$-module.   Using the basis $\beta$, we can write $\varrho_\ell$ as a representation
\[
\varphi_\ell \colon \Gal_\QQ \to \GL_2(\OO/\ell \OO).
\]
By Theorem~\ref{T:final}, we find that $\varphi_\ell(\Gal_{\QQ(i)}) = \SL_2(\FF_\ell)$ and that $\varphi_\ell(\Gal_\QQ)$ is generated by $\SL_2(\FF_\ell)$ and the scalar matrix $\bbar{i} I$ where $\bbar{i}$ is the image of $i$ in $\OO/\ell \OO$.

Let $\lambda$ be a prime ideal of $\OO$ lying over $\ell$ and set $\FF_\lambda=\OO/\lambda$.   Composing $\varphi_\ell$ with the reduction modulo $\lambda$ map gives a representation
\[
\varphi_\lambda \colon \Gal_\QQ \to \GL_2(\FF_\lambda).
\]
We have $\varphi_\lambda(\Gal_{\QQ(i)}) = \SL_2(\FF_\ell)$ and the group $\varphi_\lambda(\Gal_\QQ)$ is generated by $\SL_2(\FF_\ell)$ and the scalar matrix $\bbar{i} I$ where $\bbar{i}$ is the image of $i$ in $\FF_\lambda$.   Note that the group $\varphi_\lambda(\Gal_\QQ)/\langle \bbar{i} I \rangle$ is isomorphic to $\PSL_2(\FF_\ell)$.

The representation $\varphi_\lambda$ is absolutely irreducible since $\varphi_\lambda(\Gal_{\QQ(i)}) = \SL_2(\FF_\ell)$. For all $\sigma\in \Gal_\QQ - \Gal_{\QQ(i)}$, we have $\det(\varphi_\lambda(\sigma)) = \bbar{i}^2=-1$.  In particular, $\det(\varphi_\lambda(c))=-1$ for any $c\in \Gal_\QQ$ that arises from complex conjugation under some embedding $\Qbar \hookrightarrow \CC$.  By Serre's modularity theorem \cite{MR885783}, which was proved by Khare and Wintenberger \cite{Khare:2009}, we find that the representation $\varphi_\lambda$ arises from a cuspidal eigenform $f$.

We shall not describe $f$ here (it can be chosen to have weight $3$ and independent of $\lambda$ and $\ell\geq 11$).  Motivated by this paper, we will discuss in future work how to obtain other $\PSL_2(\FF_\ell)$ extensions of $\QQ$ using suitable eigenforms.  

\appendix
\section{Proof of Lemmas from \S\ref{S:basics}} \label{SS:appendix}

Fix an odd prime $\ell$.   For background on \'etale cohomology, see \cite{MR559531}.

\subsection{Galois representations}

For each positive integer $n$, let $E[\ell^n]$ be the $\ell^n$-torsion subscheme of $E$; it is a sheaf of $\ZZ/\ell^n\ZZ$-modules on $U$ which is free of rank $2$.  The sheaves $\{E[\ell^n]\}_{n\geq 1}$ with the multiplication by $\ell$ morphisms $E[\ell^{n+1}]\to E[\ell^n]$ form a sheaf of $\ZZ_\ell$-modules on $U$ which we denote by $T_\ell(E)$.    

Let $\eta$ be the generic point of $U$.   Set $K=\Qbar(t)$.   Fix an algebraic closure $\Kbar$ of $\QQ(t)$ containing $K$ and let $\bbar{\eta}$ be the corresponding geometric generic point of $U$.  The sheaf $E[\ell^n]$ on $U$ corresponds to a representation
\[
\beta_{\ell^n} \colon \pi_1(U,\bbar{\eta}) \to \Aut_{\ZZ/\ell^n\ZZ}(E[\ell^n]_{\bbar\eta})\cong \GL_2(\ZZ/\ell^n\ZZ).
\]
Let $E_\eta$ be fiber of $E\to U$ above $\eta$; it is the elliptic curve over $\QQ(t)$ defined by (\ref{E:main}).
We can identify the stalk $E[\ell^n]_{\bbar\eta}$ with the group of $\ell^n$-torsion points in $E_{\eta}(\Kbar)$.  The representation $\beta_{\ell^n}|_{\pi_1(U_\Qbar,\bbar\eta)}$ extends to a representation $\tilde\beta_{\ell^n}\colon \Gal(\Kbar/K) \to \GL_2(\ZZ/\ell^n\ZZ)$.

\begin{lemma} \label{L:big monodromy}
For all $n\geq 1$, we have $\beta_{\ell^n}(\pi_1(U_\Qbar,\bbar{\eta}))=\SL_2(\ZZ/\ell^n\ZZ)$.
\end{lemma}
\begin{proof}
It suffices to show that $\tilde\beta_{\ell^n}(\Gal(\Kbar/K))=\SL_2(\ZZ/\ell^n\ZZ)$.  This follows from Proposition~2.1 of \cite{MR757503} and requires that $\ell$ is odd; note that the $j$-invariant of $E_\eta/\QQ(t)$ is $256(t^4-t^2+1)^3 t^{-4}(t-1)^{-2}(t+1)^{-2}$ and hence the least common multiple of the order of its poles is $4$. 
\end{proof}

\begin{lemma} \label{L:H2 vanish}
For all $n\geq 1$ and $i\neq 1$, we have $H^i_c(U_\Qbar, E[\ell^n])=0$.
\end{lemma}
\begin{proof}
The lemma holds for $i \geq 3$ since $U$ has dimension $1$.  The lemma holds for $i=0$ since $U$ is an affine curve.  We may now assume that $i=2$.
By Poincar\'e duality, it suffices to show that $H^0_{\et}(U_\Qbar, E[\ell^n]^\vee(1))=0$.  The Weil pairing shows that $E[\ell^n]^\vee(1)$ and $E[\ell^n]$ are isomorphic, so we need only show that $H^0_{\et}(U_\Qbar, E[\ell^n])=0$.   We have $H^0_{\et}(U_\Qbar, E[\ell^n]) = E[\ell^n]_{\bbar\eta}^{\pi_1(U,\bbar\eta)}$ which is $0$ by Lemma~\ref{L:big monodromy}.
\end{proof}

Let $j\colon U\hookrightarrow \PP^1_\QQ$ be the inclusion morphism.

\begin{lemma} \label{L:dim 4}
\begin{romanenum}
\item \label{L:dim 4 a}
For $n\geq 1$, we have $j_!(E[\ell^n])=j_*(E[\ell^n])$.   We have $j_!(T_\ell(E))=j_*(T_\ell(E))$.
\item \label{L:dim 4 b}
The $\FF_\ell$-vector space $V_\ell$ has dimension $4$.
\end{romanenum}
\end{lemma}
\begin{proof}
Take any closed point $x$ of $\PP^1_\Qbar$ and let $I_x$ be a corresponding inertia subgroup of $\Gal(\Kbar/K)$.    The stalk $j_*(E[\ell^n])_x$ is isomorphic to $E[\ell^n]_{\bbar\eta}^{I_x}$.  If $E_\eta$ has good, multiplicative or additive reduction at $x$, then the $\ZZ/\ell^n\ZZ$-module $j_*(\calF)_x$ is free of rank $2$, $1$ or $0$, respectively.   The elliptic curve $E_\eta$ has good reduction at the points of $U_\Qbar$ and additive reduction at the closed points of $\PP^1_\Qbar-U_\Qbar$ (moreover, it has reduction of Kodaira type $I_2^*$ or $I_4^*$ at $0$, $1$, $-1$ and $\infty$).   So $j_*(E[\ell^n])$ vanishes on $\PP^1_\Qbar-U_\Qbar$ and thus  $j_!(E[\ell^n])=j_*(E[\ell^n])$.   This implies that $j_!(T_\ell(E))=j_*(T_\ell(E))$ which completes the proof of (\ref{L:dim 4 a}).

Now restrict to the case $n=1$.  Define $\chi := \sum_i (-1)^i \dim_{\FF_\ell} H^i_{\et}(\PP^1_\Qbar, j_*(E[\ell]))$.  By \cite{MR559531}*{V~Theorem~2.12}, we have
\[
\chi = (2-2\cdot 0) \cdot 2 - {\sum}_x c_x
\]
where the sum is over the closed points $x$ of $\PP^1_\Qbar$ and $c_x$ is the exponent of the conductor of $j_*(E[\ell])$ at $x$.   The integer $c_x$ can be computed using Tate's algorithm \cite{MR0393039}.   Since our fields have characteristic $0$, we have $c_x = 2 -  d_x$ where $d_x$ is the dimension over $\FF_\ell$ of the stalk of $j_*(E[\ell])$ at $x$.    Therefore, $c_x=0$ if $x$ belongs to $U_\Qbar$ and $c_x=2$ otherwise.   Therefore, $\chi = 4- 4\cdot 2 = -4$.

By part (\ref{L:dim 4 a}) and Lemma~\ref{L:H2 vanish}, we have 
\[
\chi=\sum_i (-1)^i \dim_{\FF_\ell} H^i_{c}(U_\Qbar, E[\ell]) = - \dim_{\FF_\ell} H^1_{c}(U_\Qbar, E[\ell]) = -\dim_{\FF_\ell} V_\ell.
\]   
Therefore, $\dim_{\FF_\ell} V_\ell = -\chi = 4$, which proves (\ref{L:dim 4 b}).
\end{proof}

Define the  $\ZZ_\ell$-module $M_\ell:= H^1_c(U_\Qbar,T_\ell(E))$ and let 
\[
\rho_{\ell^\infty} \colon \Gal_\QQ \to \Aut_{\ZZ_\ell}(M_\ell)
\]
be the representation describing the natural Galois action on $M_\ell$. 

\begin{lemma} \label{L:M info}
\begin{romanenum}
\item \label{L:M info a}
The $\ZZ_\ell$-module $M_\ell$ is free of rank $4$.
\item \label{L:M info b}
The homomorphism $M_\ell \to V_\ell$ induced by the morphism $T_\ell(E)\to E[\ell]$ gives an isomorphism between $M_\ell/\ell M_\ell$ and $V_\ell$ that respects the $\Gal_\QQ$-actions.
\end{romanenum}
\end{lemma}
\begin{proof}
The short exact sequence $0 \to T_\ell(E) \xrightarrow{\times\ell} T_\ell(E) \to E[\ell] \to 0$ of sheaves on $U$ gives rise to an exact sequence
\[
H^0_c(U_\Qbar, E[\ell])  \to M_\ell \xrightarrow{\times \ell} M_\ell \to V_\ell \to H^2_c(U_\Qbar, T_\ell(E)).
\]
We have $H^0_c(U_\Qbar, E[\ell])=0$ and $H^2_c(U_\Qbar, T_\ell(E))=\varprojlim_n H^2_c(U_\Qbar, E[\ell^n]) =0$ by Lemma~\ref{L:H2 vanish}.   The short exact sequence $0\to M_\ell \xrightarrow{\times \ell} M_\ell \to V_\ell \to 0$ then gives the desired isomorphism of part (\ref{L:M info b}).   Since multiplication by $\ell$ on $M_\ell$ is injective, we deduce that the $\ZZ_\ell$-module $M_\ell$ is free.    The rank of $M_\ell$ is then equal to the dimension of $M_\ell/\ell M_\ell =V_\ell$, which is $4$ by Lemma~\ref{L:dim 4}.
\end{proof}

Define the ring $R= \ZZ[1/2,1/\ell]$ and the $R$-scheme 
\[
\calU:= \Aff^1_{R}-\{0,1,-1\}=\Spec R[t,(t(t-1)(t+1))^{-1}].
\]  
The equation (\ref{E:main}) defines a relative elliptic curve $\calE\to \calU$.    For each positive integer $n$, let $\calE[\ell^n]$ be the $\ell^n$-torsion subscheme of $\calE$; it is a sheaf of $\ZZ/\ell^n\ZZ$-modules on $\calU$ which is free of rank $2$.    The sheaves $\{\calE[\ell^n]\}_{n\geq 1}$ with the multiplication by $\ell$ morphisms $\calE[\ell^{n+1}]\to \calE[\ell^n]$ form a sheaf of $\ZZ_\ell$-modules on $\calU$ which we denote by $T_\ell(\calE)$.    

Let $j\colon \calU\hookrightarrow \PP^1_{R}$ be the inclusion morphism and let $\pi\colon \PP^1_{R}\to \Spec R$ be the structure map.  Define the sheaf $\calF:= R^1 \pi_*( j_*(T_\ell(\calE)))$ on $\Spec R$ of $\ZZ_\ell$-modules.    Using \cite{MR0354654}*{XVI~Corollaire~2.2}, we find that the sheaf $\calF$ on $\Spec R$ is a lisse sheaf of $\ZZ_\ell$-modules.  The sheaf $\calF$ thus corresponds to a representation $\varrho_{\ell^\infty} \colon \pi_1(\Spec R, \bbar\xi) \to  \Aut_{\ZZ_\ell}( \calF_{\bbar\xi} )$ where $\bbar\xi$ is the geometric point of $\Spec R$ corresponding to a fixed algebraic closure $\Qbar$.  

Base extension gives a morphism $\calE_\QQ \to \calU_\QQ$ that agrees with our earlier morphism $f\colon E\to U$.  The restriction of $T_\ell(\calE)$ to $U=\calU_\QQ$ is our sheaf $T_\ell(E)$.   Therefore, $\calF_{\bbar\xi} = H^1_{\et}(\PP^1_\Qbar, j_*(T_\ell(\calE)))=H^1_{\et}(\PP^1_\Qbar, j_*(T_\ell(E)))$ which is $H^1_c(U_\Qbar, T_\ell(E))=M_\ell$ by Lemma~\ref{L:dim 4}(\ref{L:dim 4 a}).  So we can view $\varrho_{\ell^\infty}$ as a morphism
\[
\varrho_{\ell^\infty} \colon \pi_1(\Spec R, \bbar\xi) \to  \Aut_{\ZZ_\ell}( M_\ell );
\]
it gives rise to a representation $\Gal_\QQ \to \Aut_{\ZZ_\ell}(M_\ell)$ that agrees with $\rho_{\ell^\infty}$ up to conjugacy.  In particular, $\rho_{\ell^\infty}$ is unramified at every prime $p\nmid 2\ell$.\\

We now fix a prime $p\nmid 2\ell$.   Let $s$ be the closed point of $\Spec R$ corresponding to the prime $p$ and let $\bbar{s}$ be the geometric point arising from a fixed algebraic closure $\FFbar_p$.     

Base extending by $s$, we obtain a relative elliptic curve $\calE_{\FF_p} \to \calU_{\FF_p}=:\calU_p$.  The fiber of $\calE_{\FF_p} \to \calU_p$ over the generic point of $\calU_p$ is the elliptic curve $E_p/\FF_p(t)$ of \S\ref{SS:compatibility}.    For each integer $n\geq 1$, let $\calE_{\FF_p}[\ell^n]$ be the $\ell^n$-torsion subscheme of $\calE_{\FF_p}$; it is a sheaf of $\ZZ/\ell^n\ZZ$-modules on $\calU_p$ which is free of rank $2$.  The sheaves $\{\calE_{\FF_p}[\ell^n]\}_{n\geq 1}$ with the multiplication by $\ell$ morphisms $\calE_{\FF_p}[\ell^{n+1}]\to \calE_{\FF_p}[\ell]$ form a sheaf of $\ZZ_\ell$-modules on $\calU_p$ which we denote by $T_\ell(\calE_{\FF_p})$.    The sheaf $T_\ell(\calE_{\FF_p})$ is of course the restriction of $T_\ell(\calE)$ to $\calU_p$.

Take any closed point $x$ of $\calU_p$ and let $\bbar{x}$ be a geometric point extending $x$.   Denote the stalk of $T_{\ell}(\calE_{\FF_p})$ at $\bbar{x}$ by $T_{\ell}(\calE_{\FF_p})_{\bbar{x}}$; it is a free $\ZZ_\ell$-module of rank $2$.    The \emph{geometric} Frobenius $F_x$ acts $\ZZ_\ell$-linearly on the fiber $T_{\ell}(\calE_{\FF_p})_{\bbar{x}}$.     Let $\Frob_x:= F_x^{-1}$ be the arithmetic Frobenius.  One can show that $\det(I- \Frob_x T^{\deg x} |  T_{\ell}(\calE_{\FF_p})_{\bbar{x}})$ equals $L_x(T):= 1 - a_xT^{\deg x}+ p^{\deg x} T^{2\deg x}$ with notation as in \S\ref{SS:compatibility}.  Therefore, $\det(I- F_x T^{\deg x} |  T_{\ell}(\calE_{\FF_p})_{\bbar{x}})$ equals
\[
\det(F_x) T^{2\deg x}\det(I- \Frob_x T^{-\deg x} |  T_{\ell}(\calE_{\FF_p})_{\bbar{x}})= p^{-\deg x} T^{2\deg x} L_x(T^{-1}) = L_x(T/p).
\]
Therefore, $L(T/p,E_p)= \prod_x L_x(T/p)= \prod_x \det(I- F_x T^{\deg x} |  T_{\ell}(\calE_{\FF_p})_{\bbar{x}})^{-1}$ where the product is over the closed points of $\calU_p$.    By the Grothendieck-Lefschetz trace formula, we have
\[
L(T/p,E_p) = {\prod}_i \det(I - \Frob_p^{-1} T | \, H^i_c(\calU_{p,\, \FFbar_p},T_\ell(\calE_{\FF_p}))\otimes_{\ZZ_\ell} \QQ_\ell)^{(-1)^{i+1}}.
\]
The morphism $j$ gives an inclusion morphism $\calU_p\hookrightarrow \PP^1_{\FF_p}$ that we also denote by $j$.   Since $E_p/\FF_p(t)$ has additive reduction at $0$, $1$, $-1$ and $\infty$, we have $j_!(T_\ell(\calE_{\FF_p}))=j_*(T_\ell(\calE_{\FF_p}))$.  Therefore, 
\[
L(T/p,E_p)= {\prod}_i \det(I - \Frob_p^{-1} T | \, H^i_{\et}(\PP^1_{\FFbar_p},j_*(T_\ell(\calE_{\FF_p})))\otimes_{\ZZ_\ell} \QQ_\ell)^{(-1)^{i+1}}.
\]

\begin{lemma} \label{L:App rho inf}
We have $L(T/p,E_p) = \det(I-\rho_{\ell^\infty}(\Frob_p^{-1}) T)$.
\end{lemma}
\begin{proof}
Take any integer $i$ and define the sheaf $\calG= R^i \pi_*(j_*(T_\ell(\calE))$ on $\Spec R$ (we have $\calF=\calG$ when $i=1$).  Using \cite{MR0354654}*{XVI~Corollaire~2.2}, we find that the sheaf $\calG$ on $\Spec R$ is a lisse sheaf of $\ZZ_\ell$-modules (and so is constructible and locally constant).  The stalk of $\calG$ at $\bbar{s}$ is $H^i_{\et}(\PP^1_{\FF_p}, j_*(T_\ell(\calE))=H^i_{\et}(\PP^1_{\FF_p}, j_*(T_\ell(\calE_{\FF_p}))$.  The stalk of $\calG$ at $\bbar\xi$ is $H^i_{\et}(\PP^1_\Qbar,  j_*(T_\ell(\calE))) = H^i_{\et}(\PP^1_\Qbar,  j_*(T_\ell(E)))$.  

If $i\neq 1$, then we have $\calG_{\bbar{\xi}}=0$ by Lemma~\ref{L:H2 vanish}, and hence $H^i_{\et}(\PP^1_{\FF_p}, j_*(T_\ell(\calE_{\FF_p}))=0$ since $\calG$ is locally constant.   Therefore,
\[
L(T/p,E_p) = \det(I - \Frob_p^{-1} T | \, H^1_{\et}(\PP^1_{\FFbar_p},j_*(T_\ell(\calE_{\FF_p})))\otimes_{\ZZ_\ell} \QQ_\ell) = \det(I - \Frob_p^{-1} T | \, \calF_{\bbar{s}} \otimes_{\ZZ_\ell} \QQ_\ell)
\]

Since $\calF$ is locally compatible and constructible, each cospecialization map $\calF_{\bbar s} \to \calF_{\bbar\xi}$ is an isomorphism of $\ZZ_\ell$-modules.   The isomorphism coming from the cospecialization map has a compatible Galois action, i.e., there is a Frobenius automorphism $\Frob_p\in \Gal_\QQ$ such that the action of $\Frob_p$ on $\calF_{\bbar{\xi}}$ corresponds with the action of $\Frob_p \in \Gal(\FFbar_p/\FF_p)$ on $\calF_{\bbar{s}}$.  In particular, $\det(I - \Frob_p^{-1} T | \calF_{\bbar\xi} \otimes_{\ZZ_\ell} \QQ_\ell )=\det(I-\Frob_p^{-1} T | \calF_{\bbar{s}} \otimes_{\ZZ_\ell} \QQ_\ell)$.   Therefore, $L(T/p)$ agrees with $\det(I - \Frob_p^{-1} T | \calF_{\bbar\xi} \otimes_{\ZZ_\ell} \QQ_\ell ) = \det(I-\rho_{\ell^\infty}(\Frob_p^{-1}) T)$.
\end{proof}

Since $M_\ell$ is a free $\ZZ_\ell$-module of rank $4$ by Lemma~\ref{L:M info}(\ref{L:M info a}), Lemma~\ref{L:App rho inf} implies that $L(T,E_p)$ is a polynomial of degree $4$; that it is has integer coefficients and is independent of $\ell$ is clear from the series definition of $L(T,E_p)$.    Define $P_p(T):= L(T/p, E_p)$.

\begin{lemma} \label{L:Pp new}
The series $P_p(T)$ is a polynomial of degree $4$ with coefficients in $\ZZ[1/p]$.  \qed
\end{lemma}

\begin{lemma} \label{L:inv compatibility}
For each prime $p\nmid 2\ell$, the representation $\rho_\ell$ is unramified at $p$ and satisfies $\det(I-\rho_\ell(\Frob_p^{-1}) T)\equiv P_p(T) \pmod{\ell}$.
\end{lemma}
\begin{proof}
This follows immediately from Lemmas~\ref{L:App rho inf} and \ref{L:M info}.
\end{proof}

\subsection{Proof of Lemma~\ref{L:basics}}

Part (\ref{I:dim 4}) is Lemma~\ref{L:dim 4}(\ref{L:dim 4 b}).  We now prove (\ref{I:pairing}).  The Weil pairing gives an alternating non-degenerate pairing $E[\ell] \times E[\ell] \to \FF_\ell(1)$ of sheaves on $U$.    It extends to an alternating non-degenerate pairing
\begin{equation} \label{E:F pairing}
j_*(E[\ell]) \times j_*(E[\ell]) \to j_*(\FF_\ell(1)) \cong \FF_\ell(1).
\end{equation}
Composing the cup product 
\[
H^1_{\et}(\PP^1_{\Qbar},j_*(E[\ell])  ) \times H^1_{\et}(\PP^1_{\Qbar}, j_*(E[\ell])  ) \xrightarrow{\cup} H^2_{\et}(\PP^1_{\Qbar},j_*(E[\ell])  \otimes j_*(E[\ell]) )
\]
with the homomorphism $H^2_{\et}(\PP^1_{\Qbar},j_*(E[\ell])  \otimes j_*(E[\ell]) ) \to H^2_{\et}(\PP^1_{\Qbar}, \FF_\ell(1)) \cong \FF_\ell$ arising from (\ref{E:F pairing}), we obtain a pairing
\[
\ang{\,}{\,} \colon V_\ell \times V_\ell \to \FF_\ell;
\]
note that $V_\ell=H^1_{\et}(\PP^1_{\Qbar},j_*(E[\ell])$ by Lemma~\ref{L:dim 4}.  We have $\ang{\sigma(v)}{\sigma(w)} = \sigma( \ang{v}{w})=\ang{v}{w}$ for all $v,w\in V_\ell$ and $\sigma\in \Gal_\QQ$.  The pairing $\ang{\,}{\,}$ is symmetric; observe that the cup-product and (\ref{E:F pairing}) are both alternating.   The pairing $\ang{\,}{\,}$ is also non-degenerate; this follows from Poincar\'e duality, cf.~\cite{MR559531}*{V~Proposition~2.2(b)}.   

\subsection{Proof of Lemmas~\ref{L:SO image}, \ref{L:deg 4} and \ref{L:explicit}}

\begin{lemma} \label{L:FE basic}
\begin{romanenum}
\item
For each odd prime $p$, we have $T^4 P_p(1/T)= \varepsilon_p \cdot P_p(T)$ for a unique $\varepsilon_p \in \{\pm 1\}$.
\item
For each odd prime $\ell$ and prime $p\nmid 2\ell$, we have $\det(\rho_\ell(\Frob_p))=\varepsilon_p$.
\end{romanenum}
\end{lemma}
\begin{proof}
Recall that $\rho_\ell(\Gal_\QQ)\subseteq \Or(V_\ell)$ by Lemma~\ref{L:basics}(\ref{I:pairing}) and $\dim_{\FF_\ell} V_\ell =4$.   So for each $A\in \Or(V_\ell)$, we have $T^4\det(I-A T^{-1}) = \det(A)\cdot \det(I-AT)$ where $\det(A)\in\{\pm 1\}$.    In particular, for each prime $p\nmid 2\ell$, we have $T^4\det(I-\rho_\ell(\Frob_p^{-1}) T^{-1}) = \det(\rho_\ell(\Frob_p^{-1})) \cdot \det(I-\rho_\ell(\Frob_p^{-1}) T)$.  By Lemma~\ref{L:inv compatibility}, we have $T^4 P_p(T^{-1}) \equiv \det(\rho_\ell(\Frob_p)^{-1})\cdot  P_p(T) \pmod{\ell}$.   

Since $T^4 P_p(T^{-1})$ is congruent to $P_p(T)$ or $-P_p(T)$ for infinitely many primes $\ell$, we have $T^4 P_p(T) = \varepsilon_p P_p(T)$ for a unique $\varepsilon_p \in \{\pm 1\}$.   Reducing modulo $\ell$, we find that $\det(\rho_\ell(\Frob_p))= \varepsilon_p^{-1}=\varepsilon_p$.
\end{proof}

For each odd prime $\ell$, define the character  $\alpha_\ell \colon \Gal_\QQ \to \{\pm 1\}$, $\sigma \mapsto \det(\rho_{\ell}(\sigma^{-1}))$.  By Lemma~\ref{L:FE basic}, we have $\alpha_\ell(\Frob_p)= \varepsilon_p$ for all $p\nmid 2\ell$.  By the Chebotarev density theorem, we find that $\alpha_\ell$ is a character $\alpha\colon \Gal_\QQ \to \{\pm 1\}$ that does not depend on $\ell$.   Since $\alpha=\alpha_\ell$ is unramified at $p\nmid 2\ell$ for all odd $\ell$, we deduce that $\alpha$ is unramified at all odd $p$.      Let $K$ be the fixed field of $\ker(\alpha)$ in $\Qbar$; it is unramified at odd primes and satisfies $[K:\QQ]\leq 2$.   So $K$ is $\QQ$, $\QQ(i)$, $\QQ(\sqrt{2})$ and $\QQ(\sqrt{-2})$.

By Lemma~\ref{L:Pp new}, $P_p(T)$ is a polynomial of degree $4$.  So to prove Lemma~\ref{L:SO image} and \ref{L:deg 4}, it suffices to show that $\alpha=1$.  If $\alpha \neq 1$, then $K\neq \QQ$ and hence $3$ or $5$ is inert in $K\in\{\QQ(i),\QQ(\sqrt{2}),\QQ(\sqrt{-2})\}$.    So if $\alpha\neq 1$, then $\varepsilon_3=\alpha(\Frob_3)$ or $\varepsilon_5=\alpha(\Frob_5)$ is $-1$.  Therefore, Lemmas~\ref{L:SO image} and \ref{L:deg 4} will follow from Lemma~\ref{L:explicit}.\\

We now sketch Lemma~\ref{L:explicit}.  Computing the product (\ref{L:L ps defn}) with the closed points $x$ with $\deg x \leq 2$, we find that $L(T,E_3)= 1 -2 T^2 +O(T^3)$ and $L(T,E_5)=1-4T+54T^2 +O(T^3)$.  Therefore, $P_3(T)= 1 -2/9\cdot  T^2 +O(T^3)$ and $P_5(T)=1-4/5\cdot T+54/25\cdot T^2 +O(T^3)$.    Since the coefficients of $T^2$ in $P_3(T)$ and $P_5(T)$ are non-zero, we have $\varepsilon_3=\varepsilon_5=1$ and hence the polynomials $P_3(T)$ and $P_5(T)$ are reciprocal of degree $4$.   Therefore, $P_3(T)= 1 -2/9\cdot  T^2 + T^4$ and $P_5(T)=1-4/5\cdot T+54/25\cdot T^2 -4/5\cdot T^3 +T^4 = (1 - 2/5\cdot T + T^2)^2$.   We have also verified these examples with \texttt{Magma}'s function \texttt{LFunction} \cite{Magma}.

\subsection{Proof of Lemma~\ref{L:compatibility}} 
Take any prime $p\nmid 2\ell$.  By Lemma~\ref{L:inv compatibility}, $\rho_\ell$ is unramified at $p$ and  $\det(I-\rho_\ell(\Frob_p^{-1}) T)\equiv P_p(T) \pmod{\ell}$.     We have $\det(I-A^{-1} T) = \det(I-A)$ for all $A\in \SO(V_\ell)$.  So by Lemma~\ref{L:SO image}, we have $\det(I-\rho_\ell(\Frob_p) T)\equiv \det(I-\rho_\ell(\Frob_p^{-1}) T)\equiv P_p(T) \pmod{\ell}$.

\subsection{Proof of Lemma~\ref{L:X connection}}

Let $\NS(X)$ be the N\'eron-Severi group of $X_{\Qbar}$.    Let $\calT$ be the subgroup of $\NS(X)$ generated by a section of $\tilde{f}$, a non-singular fiber of $\tilde{f}$ over a rational point of $\PP^1_\QQ$, and the irreducible components of the singular fibers of $\tilde{f}$.    There is a natural $\Gal_\QQ$-action on $\NS(X)$ that preserves $\calT$.

\begin{lemma} \label{L:Galois on singular fibers}
\begin{romanenum}
\item \label{L:Galois on singular fibers a}
The group $\calT$ is free abelian of rank $30$.
\item \label{L:Galois on singular fibers b}
The group $\Gal_\QQ$ acts trivially on $\calT$.
\end{romanenum}
\end{lemma}
\begin{proof}
The singular fibers of $\tilde f \colon X \to \PP^1_\QQ$ are at $0$, $1$, $-1$ and $\infty$.   Fix $s\in \{0,1,-1,\infty\}$, and let $\calI_s$ be the set of irreducible components (defined over $\Qbar$) of the fiber $\tilde f^{-1}(s)$ of $\tilde{f}$ over $s$.   There is a natural action of $\Gal_\QQ$ on $\calI_s$; to prove (\ref{L:Galois on singular fibers b}) it suffices to show that this action is trivial.    

By Tate's algorithm, the fiber $\tilde{f}^{-1}(s)$ of $\tilde{f}$ at $s$ is of Kodaira type $\text{I}^*_4$ if $s\in \{0,\infty\}$ and of Kodaira type $\text{I}^*_2$ if $s\in\{1,-1\}$.   The Dykin diagram corresponding to the intersection matrix of $\calI_s$ is of type $D_6$ or $D_8$.   We find that $\Gal_\QQ$ acts trivially on $\calI_s$ if and only if $\Gal_\QQ$ acts trivially on the $4$ elements of $\calI_s$ that occur with multiplicity $1$ in $\tilde{f}^{-1}(s)$.    Using Tate's algorithm, one can show that the four irreducible components of $\tilde{f}^{-1}(s)$ with multiplicity $1$ are all defined over $\QQ$.

Part (\ref{L:Galois on singular fibers a}) follows from equation (1.3) of \cite{MR1211006}.  
\end{proof}

 Let $\gamma\colon  \NS(X) \to H^2_{\et}(X_\Qbar, \QQ_\ell(1))$ be the cycle map; it is a homomorphism that respects the natural $\Gal_\QQ$-actions.   As explained in \S1 of \cite{MR1211006}, there is a $\QQ_\ell[\Gal_\QQ]$-submodule $\calV_\ell$ such that 
 \begin{equation} \label{E:H2 last}
 H^2_{\et}(X_\Qbar, \QQ_\ell(1)) =  \calV_\ell \oplus (\gamma(\calT) \otimes_{\ZZ_\ell} \QQ_\ell) \cong \calV_\ell \oplus \QQ_\ell^{30};
 \end{equation}
 the last isomorphism uses Lemma~\ref{L:Galois on singular fibers}.  By Proposition~1 of \cite{MR1211006}, the vector space $\calV_\ell$ has dimension $4$ over $\QQ_\ell$.
 
\begin{lemma} \label{L:free over Zl}
The $\ZZ_\ell$-module $H^2_{\et}(X_\Qbar,\ZZ_\ell)$ is free.
\end{lemma}
\begin{proof}
By choosing an embedding $\Qbar\hookrightarrow \CC$ and using the comparison isomorphism, we need only show that $H^2(X(\CC),\ZZ_\ell)=H^2(X(\CC),\ZZ)\otimes_\ZZ \ZZ_\ell$ is a free $\ZZ_\ell$-module.  By applying Corollary~1.48 of \cite{MR538682} to the elliptic fibration $X(\CC)\to \PP^1(\CC)$ arising from $\tilde f$, we find that $H^2(X(\CC),\ZZ)$ is a free abelian group. 
\end{proof} 

By Lemma~\ref{L:free over Zl}, the $\ZZ_\ell$-module $H^2_{\et}(X_\Qbar,\ZZ_\ell(1))$ is free.   So from (\ref{E:H2 last}), we deduce that the semi-simplification of $H^2_{\et}(X_\Qbar, \FF_\ell(1))$ as an $\FF_\ell[\Gal_\QQ]$ is isomorphic to $\bbar{\calV}_\ell \oplus \FF_\ell^{30}$  where $\bbar{\calV}_\ell$ has dimension $4$ over $\FF_\ell$.   The lemma is then a consequence of the following.

\begin{lemma}
The $\FF_\ell[\Gal_\QQ]$-modules $V_\ell$ and $\bbar{\calV}_\ell$ are isomorphic.
\end{lemma}
\begin{proof}
Using Lemma~\ref{L:explicit}, we have $\det(I - \rho_{\ell}(\Frob_3)) \equiv 1 -2/9 \cdot T^2 + T^4 \pmod{\ell}$.   The discriminant of $1 -2/9 \cdot T^2 + T^4$ is $2^{16} 5^2/3^8$.  Since $\ell\geq 7$, the polynomial $1 -2/9 \cdot T^2 + T^4 \in \FF_\ell[T]$ is separable and $1$ is not a root.   Therefore, $V_\ell$ is a semi-simple $\FF_\ell[\Gal_\QQ]$-module and $V_\ell^{\Gal_\QQ}=0$.   Since also $\dim_{\FF_\ell} V_\ell =4$, to prove the lemma we need only show that $V_\ell$ is the quotient of some $\FF_\ell[\Gal_\QQ]$-submodule of $H^2_{\et}(X_\Qbar,\FF_\ell(1))$.

Consider the Leray spectral sequence with $E_2$-terms $E_2^{p,q}= H^p_{\et}(\PP^1_{\Qbar}, R^q\tilde f_* \FF_\ell(1))$ which converges to $L^n= H^n_{\et}(X_\Qbar, \FF_\ell(1))$.    Let $E_\infty^{p,q}$ be the limiting value of the $\{E^{p,q}_r\}_r$.   There is thus a filtration $0\subseteq W_1 \subseteq W_2 \subseteq L^2=H^2_{\et}(X_\Qbar,\FF_\ell(1))$ of $\FF_\ell[\Gal_{\QQ}]$-modules such that $W_2/W_1$ is isomorphic to $E^{1,1}_\infty$.   Using that $E_2^{p,q}=0$ for $p>2$, we find that $E^{1,1}_\infty$ equals $E^{1,1}_2=H^1_{\et}(\PP^1_{\Qbar}, R^1\tilde f_*\FF_\ell(1))$.    So we need only show that the $\FF_\ell[\Gal_\QQ]$-modules $V_\ell$ and $H^1_{\et}(\PP^1_{\Qbar}, R^1\tilde f_*\FF_\ell(1))$ are isomorphic.   In particular, it suffices to show that the sheaves $j_!(E[\ell])$ and $R^1\tilde f_*\FF_\ell(1)$ of $\FF_\ell$-modules on $\PP^1_\QQ$ are isomorphic.     

Take any geometric point $s$ of $\PP^1_\QQ - U$.   The stalk of $R^1 \tilde{f}_*\FF_\ell(1)$ at $s$ is $H^1_{\et}(X_s,\FF_\ell(1))$ where $X_s$ is the fiber of $\tilde{f}$ above $s$.   The fiber $X_s$ is simply connected (it has Kodaira type $I_2^*$ or $I_4^*$) and hence the stalk of $R^1 \tilde{f}_*\FF_\ell(1)$ at $s$ vanishes.   

It thus suffices to show that $E[\ell]$ and $R^1 f_* \FF_\ell(1)= R^1 \tilde{f}_*\FF_\ell(1)|_U$ are isomorphic sheaves of $\FF_\ell$-modules on $U$.   
Since these sheaves are both locally constant and constructible, it suffices to prove that $E[\ell]_{\bbar\eta}$ and $(R^1 f_* \FF_\ell(1))_{\bbar\eta} = H^1_{\et}(E_{\bbar\eta},\FF_\ell(1))$ are isomorphic $\FF_\ell[\Gal(\bbar{\QQ(t)}/\QQ(t))]$-modules.   Indeed, if $\calE$ is an elliptic curve over \emph{any} field $k$ of characteristic $0$, then the group of $\ell$-torsion points in $
\calE(\kbar)$ is isomorphic as an $\FF_\ell[\Gal(\kbar/k)]$-modules with $H^1_{\et}(\calE_{\kbar},\FF_\ell(1))$.
\end{proof}

\bibliographystyle{plain}

\end{document}